\documentclass[10pt,a4paper]{article}
\usepackage[latin1]{inputenc}
\usepackage{amsmath, amsthm, amssymb}
\usepackage{cleveref}
\usepackage{graphicx}

\newtheorem{theorem}{Theorem}[section]
\newtheorem{lemma}[theorem]{Lemma}

\newtheorem{prop}[theorem]{Proposition}

\newtheorem{open}{Open Problem}

\newtheorem*{apxtheorem*}{Theorem}

\DeclareMathOperator{\finitetype}{\phi-finite}
\DeclareMathOperator{\Ber}{Ber}
\DeclareMathOperator{\Exp}{Exp}

\DeclareMathOperator{\TV}{TV}
\DeclareMathOperator{\MC}{MC}
\DeclareMathOperator{\m}{m}
\DeclareMathOperator{\Agents}{Agents}

\DeclareMathOperator{\supp}{supp}
\DeclareMathOperator{\fs}{fs}
\DeclareMathOperator{\cs}{cs}
\DeclareMathOperator{\Ord}{\mathfrak{O}}
\DeclareMathOperator{\R}{\mathbb{R}}
\DeclareMathOperator{\E}{\mathbb{E}}
\DeclareMathOperator{\N}{\mathbb{N}}
\DeclareMathOperator{\SBO}{SBO}

\DeclareMathOperator{\Prb}{\mathbb{P}}
\DeclareMathOperator{\uniform}{unif}
\DeclareMathOperator{\setN}{[N]}

\begin{document}

\title{The Metric Coalescent}
\author{Daniel Lanoue}
\date{\today}

\maketitle

\begin{abstract}
The Metric Coalescent (MC) is a measure-valued Markov Process generalizing the classical Kingman Coalescent. We show how the MC arises naturally from a discrete agent based model (Compulsive Gambler) of social dynamics and prove an existence and uniqueness theorem extending the MC to the space of all Borel probability measures on any locally compact Polish space.
\end{abstract}

\tableofcontents

\section{Introduction}

Introduced in \cite{ALS14} the Compulsive Gambler (CG) process is an Finite Markov Information Exchange (FMIE) process modelling a finite set of agents meeting pairwise randomly and playing a fair winner-take-all game. The precise set-up of the model in the FMIE framework is two leveled. The first level consists of the agents and their meeting model, described by the finite set $\Agents$ and a non-negative array $(\nu_{ij})$ of meeting rates indexed by unordered pairs $\{i, j\}$ of agents. A pair of agents $i, j$ then meet at the times an independent Poisson process of rate $\nu_{ij}$.

The second level in the FMIE framework is the information exchange model, which describes the state $X_i(t)$ of each agent $i$ at time $t$ and the (deterministic or random) update rule for the states of two agents $i$ and $j$ upon meeting. In the CG process, when two agents with non-zero wealth $a$ and $b$ meet they play a fair game in which one agent acquires the combined amount $a + b$ and the other is left with nothing (thus with probabilities $\frac{a}{a + b}$ and $\frac{b}{a + b}$ respectively). Normalizing the total wealth present amongst the agents, the state of the CG process can be viewed at each time as a probability measure over $\Agents$.

Several directions for research concerning the CG process are suggested in \cite{ALS14} and this paper is devoted to a detailed study of one direction, the extension to continuous state space.  We introduce the Metric Coalescent (MC), a measure valued Markov process defined over any metric space $(S,d)$, a generalization of the CG process given by adding some geometry to the meeting model between agents. Viewing each agent as occupying a location in $S$, their meeting rates are then determined by the distances between them.

Our interest in the Metric Coalescent comes from the wide variety of techniques that can be exploited in its study. As a natural generalization of an FMIE process, we can apply tools from the study of finite Markov chains (or more generally Interacting Particle Systems) and in particular study the process dynamics of certain martingales associated to the process. As the MC is constructed from $(S, d)$ in terms of an exchangeable random distance array (as in \cite{vershik2004}), tools from the theory of exchangeability such de Finetti's Theorem and Kingman's Paintbox Theorem are applied. Finally, viewed through a symmetry breaking duality with an exchangeable partition, we make use of a comparison to the classical Kingman's Coalescent.

\subsection{Setup}

Let $(S, d)$ be an arbitrary metric space and write $P(S)$ for the space of Borel probability measures on $S$. We write $P_{\fs}(S)$ for the subspace of finitely supported measures and $P_{\cs}(S)$ for the subspace of compactly supported measures.

 We write $C(S)$ for the space of continuous functions on $S$, $C_b(S) \subset C(S)$ for the subset of bounded continuous functions and $C_0(S) \subset C_b(S)$ for the space of compactly supported functions. For $f \in C_b(S)$ and $\mu \in P(S)$, we write $\mu(f)$ for the integral
 $$
\mu(f) = \int_S f \, d \mu.
 $$
  We will view $P(S)$ as endowed with the weak topology, which we recall is given by weak convergence of measures; i.e. $\mu_i \rightarrow \mu$ weakly if and only if
$$
\mu_i(f) \rightarrow \mu(f)
$$
for all $f \in C_b(S)$. The topology of weak convergence on $P(S)$ is metrizable with the Prokhorov metric $d_P$. By Prokhorov's theorem, if $(S, d)$ is separable and complete so is $(P(S), d_P)$. \cite{billingsley2009convergence} \cite{prokhorov1956convergence}

\subsection{The Metric Coalescent}
\label{sec:define MC}

Here we will define the Metric Coalescent process $\mu_t, t \geq 0$ from any initial measure $\mu$ in $P_{\fs}(S)$. This process is essentially just a reformulation of the Compulsive Gambler process, a toy example of a Finite Markov Information Exchange process (FMIE).\cite{ALS14}

Let $\phi \colon \R_{\geq 0} \rightarrow \R_{\geq 0}$ be a fixed continuous function. For a finite measure $\mu \in P_{\fs}(S)$ we write $\# \supp \mu$ for the cardinality of the support of $\mu$. For any such $\mu$, choose arbitrarily a representation of $\mu$ as
$$
\mu = \sum_{n = 1}^{ \# \supp \mu} p_i \delta(s_i)
$$
where $s_i \in S$ for $1 \leq i \leq  \# \supp \mu$ and $p_i = \mu(s_i)$. For $1 \leq i, j \leq  \# \supp \mu$ and $i \neq j$ write 
$$
\nu_{ij} = \phi(d(s_i, s_j)),
$$
which will be the instantaneous rate of coalescence between the atoms at $s_i$ and $s_j$.

The Metric Coalescent process on $P_{\fs}(S)$ is defined as the continuous-time Markov chain with transition rates
$$
\mu \rightarrow \mu + p_i (\delta(s_j) - \delta(s_i)) 
$$
at rate $\frac{p_i}{p_i + p_j} \nu_{ij}$ for any $1 \leq i, j \leq  \# \supp \mu$. Clearly this does not depend on the choice of representation of $\mu$. Equivalently, at rate $\nu_{ij}$, atoms $s_i$ and $s_j$ 'meet' and then merge their mass either all into $s_i$ or $s_i$ with probabilities proportional to $p_i$ and $p_j$ respectively.

Our FMIE interpretation of the MC, by way of the Compulsive Gambler process, is that the sites $s_i \in \supp \mu$ are agents with wealth $p_i$. These agents then meet at rates according to their geometry in $S$ and upon meeting play a fair winner-take-all game.

The name Metric Coalescent is justified by the following theorem, which follows easily from a comparison with Kingman's Coalescent (as in \Cref{sec:Finite Support}) whenever the rate function is non-zero.

\begin{theorem}{If $\phi > 0$, for any initial measure $\mu_0 \in P_{\fs}(S)$ there is a stopping time $T < \infty$ almost surely, such that $\mu_T$ is a point mass distributed as
$$
\mu_T = \delta(\xi)
$$
where $\xi$ is distributed as $\mu_0$.}
\label{thrm:Coalescent}
\end{theorem}

\subsection{Extension to $P(S)$}

The original Metric Coalescent (MC) process as in \Cref{sec:define MC} is defined for any metric space $(S, d)$ starting from any finitely supported measure $\mu \in P_{\fs}(S)$. Our goal in this paper is to show for a complete, separable and locally compact metric space $(S, d)$ that this process can be extended to all of $P(S)$. We assume in the sequel unless otherwise stated that our metric space $(S,d)$ satisfies these conditions.

We will also need to make the following two extra assumptions on the continuous rate function $\phi$.

\begin{enumerate}
\item[(H1)] $\phi(x) > 0$ for all $x > 0$.

\item[(H2)] $\lim_{x \downarrow 0}\phi(0) = \infty$.
\end{enumerate}

The first assumption will ensure that the process ``coalesces''. Importantly we do not assume that $\phi$ is bounded away from zero. A generalization of the MC process without assuming non-zero meeting rates would be of interest; in particular in the Compulsive Gambler process on graphs, pairs of vertices (i.e. agents) only have positive meeting rates if they are connected by an edge \cite{ALS14}. However we do not consider such a generalization in this paper.

The second condition is necessary for the MC to be Feller continuous (see \Cref{sec:Examples}). Heuristically we think of our rate function as something like $\phi = \frac{1}{x}$ or $\phi = \frac{1}{x^2}$ although the class of allowable rate functions is clearly much more general.

Our main result is the following.

\begin{theorem}{There exists a unique cadlag Feller continuous $P(S)$-valued Markov process $\mu_t, t \geq 0$ defined from any initial probability measure $\mu_0 = \mu$ on $(S, d)$ such that for any $\mu \in P_{\cs}(S)$
\begin{enumerate}
\item $\mu_t \in P_{\fs}(S)$ for all $t > 0$, almost surely,
\item  For each $t_0 > 0$, the process $(\mu_t, t_0 \leq t < \infty)$ is distributed as the Metric Coalescent started at $\mu_{t_0}$.
\end{enumerate}
}
\label{thrm:Main Theorem}
\end{theorem}

There are several variants of \Cref{thrm:Main Theorem}. If the initial measure $\mu_0$ is not compactly supported, then $\mu_t$ is still \textbf{locally} finitely supported following from \Cref{prop:K finite} and the local compactness of $S$. If $\phi(\cdot)$ is bounded away from zero on $\supp \mu_0$ for some $\mu_0 \in P(S)$ then \Cref{thrm:Coalescent} holds for $\mu_0$ by the same comparison to Kingman's Coalescent. More generally, while any initial measure $\mu_0$ may not every coalesce to a point mass in finite time (for a counterexample see \Cref{sec:Finite Supp Ex}), a simple corollary of the proof of \Cref{thrm:Main Theorem} is that
$$
\lim_{t \rightarrow \infty} \mu_t = \delta(\xi),
$$
almost surely; that is, the random measure $\mu_t$ converges weakly to the random measure given by a point mass chosen from $\mu_0$.

Our method in this paper will be to first construct one such extension of the MC process to all of $P(S)$. We then prove that our construction is Feller continuous which leads to uniqueness in distribution.

We will often refer to our extension as \textbf{the} Metric Coalescent, which will be of course justified by \Cref{thrm:Main Theorem}. When we need to distinguish the two, we will often refer to the \textbf{original} Metric Coalescent (as defined in \Cref{sec:define MC}) as such.

\subsection{Overview}

In \Cref{sec:TTP}, we define the Token Process and its empirical measures and use an exchangeability argument (\Cref{prop:Token Limit}) to construct the process $\mu_t, t \geq 0$ that will ultimately be our extension of the Metric Coalescent. In \Cref{sec:Markov Property}, we prove that the process $\mu_t, t \geq 0$ is a time homogeneous Markov process.

In \Cref{sec:Finite Support}, we show that for any compact set $K \subset S$, for all positive times $t > 0$, $\mu_t$ has finite support on $K$. An immediate corollary of this is that for any compactly supported initial measure $\mu \in P_{\cs}(S)$, for all positive times $\mu_t, t > 0$ is finitely supported; in \Cref{sec:MC} it's shown that this process in $P_{\fs}(S)$ is the originally defined Metric Coalescent.

In \Cref{sec:Martingales} we examine the real-valued processes $\mu_t(f), t \geq 0$ for $f \in C_b(S)$, show that for any $f$ the process is a martingale and use a second moment argument to prove that the MC process is right continuous at $t = 0$ almost surely. In \Cref{sec:TFC} we use a coupling argument to prove that MC process is Feller continuous.

Finally, in \Cref{sec:Cont and uniq} we prove that the MC process is almost surely cadlag and use this to prove that our construction is the unique such extension of the originally defined MC.

We conclude in \Cref{sec:Examples} and \Cref{sec:Open Probs} with some motivating examples of the MC process and some open problems and further directions for research.

\section{The Token Process}
\label{sec:TTP}

Our construction of the extension of the Metric Coalescent begins with a Markov process we will refer to as the \textbf{token process}. For the Compulsive Gambler, an alternative description is provided by augmenting the process with rankings of the agents, where the random winners at meetings between agents are then deterministically replaced by the rankings and their initial randomness. The equivalence in distribution of these two descriptions of the process is a fundamental tool in its study. \cite{ALS14}

Analogously, the token process gives an alternative description of the MC that replaces the mass (or in our heuristic description ``wealth'') held by agents throughout $S$ with partitions of the natural numbers. The randomness at meeting times is replaced by a simple deterministic model driven by some initial randomness of the process. Most importantly, unlike our original definition of the Metric Coalescent, this alternative description extends naturally to any measure in $P(S)$.

In this section, we first construct the token process and then prove in \Cref{prop:Token Limit} the existence of the limit process $\mu_t, t \geq 0$ that will turn out to be the Metric Coalescent.

\subsection{Construction}
\label{sec:Token Construction}

Fix any initial measure $\mu \in P(S)$. In this section we will define the token process 
$$
\{ (u_i(t), \xi_i) \}_{t \geq 0}^{1 \leq i < \infty}
$$
 generated from $\mu$. For $i \geq 1$ choose $\xi_i$ as IID samples of $\mu$. At time $t = 0$ set $u_i(0) = i$ for all $i$ such that $\xi_i$ is unique among the other samples (the case of non-uniqueness to follow). 

For each $1 \leq i, j$ with $i \neq j$ write 
$$
\nu_{ij} = \phi(d(\xi_i, \xi_j)).
$$
We view this as the meeting rate between $i$ and $j$ and write $t_{ij}$ -- distributed as an independent rate $\nu_{ij}$ exponential random variable -- for the meeting time of $i, j$. If $\xi_i = \xi_j$ -- i.e. heuristically the rate $\nu_{ij}$ is infinite -- then we set $t_{ij} = 0$. The token process will be completely determined by the initial locations $\xi_i, i \geq 1$ and the meeting times $t_{ij}$, $i \neq j \geq 1$.

  From the meeting times $t_{ij}, i, j \geq 1$ we construct a coalescing partition process 
  $$
  Z(t) = \{ Z_i(t) \}_{i \geq 1, t \geq 0 }
  $$
   over the natural numbers $\N$. To distinguish between the elements of this partition from $\N$ used as an index set, we refer to the elements of the copy of $\N$ being partitioned as \textbf{tokens}. At time $t = 0$, set 
 $$
 Z_i(0) = \{ i \}.
 $$
  At each meeting time $t_{ij}$ between $i < j$, if $Z_i \neq \emptyset$ we set 
  $$
  Z_i(t) = Z_i(t-) \cup Z_j(t-), \hspace{2 pc} Z_j(t) = \emptyset.
  $$
  If however $Z_i = \emptyset$ or $Z_j = \emptyset$ then we ignore the meeting and call it \textbf{trivial}. At any \textbf{nontrivial} meeting time $t_{ij}$, we say that $j$ \textbf{merges} with $i$. Note that each $j \geq 1$ may merge at most once. If a token $i$ has not merged by time $t$ we say that $i$ is \textbf{alive} at time $t$.
  
 For $N \geq 1$ we will also consider the $N$-th partition processes 
 $$
 Z^N(t) = \left( Z^N_1(t), \ldots, Z^N_N(t) \right)
 $$
 given by
 $$
 Z^N_n(t) = Z_n(t) \cap \setN,
 $$
 where $\setN = \{1, 2, \ldots, N \}$.

 In the case of two or more tokens initially occupying the same site -- i.e. whenever $\xi_i = \xi_j$ for some $i \neq j$ -- we merge all the tokens at that site instantaneously into the 'lowest' token. This is consistent if more than two tokens occupy the same site, since if $$t_{ij} = t_{i j'} = 0$$ then also $t_{j j'} = 0$. The end result is that for any site $s \in S$, the lowest token $\bar{i}$ at $s$ starts at $t  = 0$ owning all the other tokens at $s$, i.e.
 $$
 Z_{\bar{i}}(0) = \{ j \colon \xi_j = \xi_{\bar{i}} \}
 $$
 and $Z_{j}(0) = \null$ for all other $j$ with $\xi_j = \xi_{\bar{i}}$.

Write $u_i(t)$ for the owner of token $i$ at time $t$, that is define $u_i(t)$ by
$$
i \in Z_{u_i(t)}.
$$
The path of token $i$ transitions as follows: if $u_i(t_{jk}-) = j$ and $k < j$ is still alive at time $t_{jk}$, then at time $t_{jk}$ we have
$$
u_i(t_{jk}) = k.
$$
Thus, the path of each token $i$ -- i.e. $u_i(t), t \geq 0$ -- depends only on the meeting times amongst tokens $1$ to $i$, showing that the infinite token process can be defined consistently. 

\subsection{The Empirical Measures}

Our main objects of focus to begin will be the random empirical measures $\mu^N_t$, given by 
$$
\mu^N_t = \frac{1}{N} \sum_{i = 1}^N \delta(\xi_{u_i(t)}), \label{def:Define muN}
$$
or equivalently 
$$
\mu^N_t = \sum_{i = 1}^N \frac{\#Z_i^N(t)}{N} \delta(\xi_i).
$$

 From the $N$-token process, via a limiting argument, we will define the stochastic process $\mu_t, t \geq 0$ in $P(S)$ that will be shown to be the extension of the Metric Coalescent. 

\begin{prop}{For any initial measure $\mu \in P(S)$, there exists a random measure-valued process $\mu_t, t \geq 0$ with $\mu_0 = \mu$ such that for all $t \geq 0$
$$
\mu^N_t \rightarrow \mu_t, \text{ almost surely.}
$$
}
\label{prop:Token Limit}
\end{prop}

In this section our main goal will be to prove this proposition.

\subsection{Exchangeable Partitions, Measure-Valued Markov Processes}

Throughout we'll rely heavily on the theory of exchangeable partitions, which we will set up. This account closely follows that in \cite{berestycki2009recent}.

A partition $\pi$ of $\N$ can be thought of either through its blocks or by the equivalence relation $\sim^{\pi}$ it induces on $\N$; i.e. $n$ and $m$ are in the same block if and only if
$$
n \sim^{\pi} m.
$$

For any permutation $\sigma$ on $\N$, we can consider the partition $\sigma \pi$ given by 
$$
i \sim^{\sigma \pi} j
$$
if and only if
$$
\sigma(i) \sim^{\pi} \sigma(j).
$$

Writing $\Pi$ for the space of partitions of $\N$, a random partition $\pi \in \Pi$ is called \textbf{exchangeable} if for any finite permutation $\sigma$ of $\N$ - i.e. constant except on finitely many $n \in \N$ - $\pi$ and $\sigma \pi$ have the same distribution on $\Pi$.

As defined in \Cref{sec:Token Construction}, the token process gives rise to the random family of partitions $Z(t), t \geq 0$ on $\N$, which we can specify by their equivalence relation $\sim_t, t \geq 0$. Importantly, note that for two tokens $i$ and $j$, we have
$$
i \sim_t j
$$
if and only if the locations of $i$ and $j$ are the same at time $t$; i.e.
$$
\xi_{u_i(t)} = \xi_{u_j(t)}.
$$

To the astute reader familiar with measure-valued Markov processes, our description of the token process is superficially very similar to the Donnelly-Kurtz Lookdown process. \cite{donnelly1999particle} The difference of course is in the details; in our setting mergers do not have corresponding birth events. More importantly, the mechanism by which tokens (corresponding loosely to levels) are selected is both dependent on the empirical measures and chosen by an exchangeable, geometrically dependent (i.e. thus not independent) meeting clock.

A closer analogy can be found between the token process and Kingman's Coalescent. In fact, this analogy will be made explicit (see \Cref{sec:Kingman}) later on, so we'll suffice for now to note that the token process can in fact be thought of as a generalization of Kingman's model, however with (geometrically defined) exchangeable rates in place of independent rates. 

\subsection{Two Types of Exchangeability}

The proof of \Cref{prop:Token Limit} will rely upon two different forms of exchangeability for the token process, corresponding to two different descriptions of the limit process $\mu_t, t \geq 0$. First, we have what we will call the \textbf{symmetric exchangeability} which will show that at any time $t$, conditional on the past (of the measure process), the token partition $Z(t)$ at time $t$ is an exchangeable partition.

The second type of exchangeability needed in our proof of \Cref{prop:Token Limit} we'll call \textbf{asymmetric}. This will show that for any time $t \geq 0$, conditional on the location of the first $K$ tokens, the location of the subsequent tokens are also exchangeable.

Both of these ultimately imply that at time $t$, the partition $Z(t)$ is exchangeable. However, both also say a bit more about that exchangeability, and combined together lead to our proof of \Cref{prop:Token Limit}.

\subsubsection{Symmetric Exchangeability}
\label{sec:Symmetric Exchg}
The first form of exchangeability is a bit harder to fully describe without making reference to $\mu_t, t \geq 0$ which currently still needs to be constructed. A more complete view of this exchangeability is given in \Cref{lem:Conditionally Exchg}.

Our main goal here is to demonstrate an approximation of this exchangeability (\Cref{prop:Uniform Partitions}) and use this to prove the following.

\begin{prop}{From any initial measure $\mu$ and for all times $t \geq 0$, the random partition process $Z(t)$ is exchangeable almost surely.}
\label{prop:Partition is Exchangeable}
\end{prop}

Recall that to show this, we'll need to show that for any time $t \geq 0$ and any finite permutation $\sigma$, the distributions of $\pi_t$ and $\sigma \pi_t$ are the same. To prove \Cref{prop:Partition is Exchangeable}, we'll focus on the changing partitions $Z^N$ of $\setN = \{1,2, \ldots, N \}$  and prove that this gives an exchangeable partition of $\setN$ for $N \geq 1$, which is clearly sufficient.

Interestingly, this proof - and thus the exchangeability of the whole partition process - relies only on how the blocks of the partition merge and not in any way on the particulars of the meeting structure (only that these meetings are exchangeable). This gives a nice example of the motivation behind the dichotomy in the framework of FMIE processes, separating the structure of a particle system into the meeting model and the information exchange model.

Let $r_1, \ldots, r_N$ be a uniform random ordering of $1, \ldots, N$. Recall the partition process $Z^N_i(t)$ defined by the token process and write $S^N_i(t)$ for
$$
S^N_i(t) = Z^N_{r_i}(t),
$$
 for $t \geq 0$, $1 \leq i \leq N$; i.e. for the tokens owned by $r_i$ at time $t$. Write $S^N(t)$ for the random partition process
 $$
 S^N(t) = (S^N_1(t), \ldots, S^N_N(t)),
 $$
 of $\setN$.
 
 For any partition $\alpha$ of $\setN$, write $|\alpha|$ for the trace $( \# \alpha_1, \ldots, \# \alpha_K)$ of the partition. We will often not specify the size, i.e. number of blocks, of a partition as in the context it will be clear.
   
  Note that if for some $i, j$ we have $\xi_{r_i} = \xi_{r_j}$, then by construction $t_{r_i r_j} = 0$ and so the merger of blocks $S^N_i$ and $S^N_j$ has ``already'' occurred at time $t = 0$. Thus, as opposed to the standard set-up of coalescing partition processes, we don't necessarily begin from a partition of all singletons. Unfortunately this serves to confuse the notation a bit but is ultimately irrelevant to our proof.
  
\begin{prop}{For any $t_0 > 0$, the conditional distribution of $S^N(t_0)$ given $\mu^N_t, 0 \leq t \leq t_0$ is uniform over all possible partitions. That is, for all $\alpha, \beta$ with $|\alpha| = |\beta|$ we have
$$
\Prb( S^N(t_0) = \alpha \vert \mu^N_t, 0 \leq t \leq t_0 ) = \Prb( S^N(t_0) = \beta \vert \mu^N_t, 0 \leq t \leq t_0).
$$}
\label{prop:Uniform Partitions}
\begin{proof}

Instead of looking at the continuous time partition process, we will consider instead the discrete jump process. Write $T_0 = 0$ and let $T_n, n \geq 1$ be the time of the $n$-th (non-trivial) meeting, i.e. the $n$-th jump of the process. Write $\mathcal{F}_n = \sigma(|S^N|(n))$ for $n \geq 0$. We will show that for $\alpha, \beta$ with $|\alpha| = |\beta|$
$$
\Prb( S^N(n) = \alpha \vert \mathcal{F}_n, \ldots, \mathcal{F}_0 ) = \Prb( S^N(n) = \beta \vert \mathcal{F}_n, \ldots, \mathcal{F}_0)
$$
for any $n \geq 0$, which is clearly equivalent to the continuous time claim. Note that $F_0$ is not trivial, as possibly singleton blocks are already merged at time $t = 0$.

Conditioning on $\mathcal{F}_n, \ldots, \mathcal{F}_0$, we know at each (non-trivial, positive) meeting which two blocks of the partition $S^N$ have merged and which of those blocks has the lowest ranking token - i.e. the ``winner'' of the meeting - but not what these tokens are. For $1 \leq k \leq n$, define $i_k$ and $j_k$ by the rule that the partition $S^N(k)$ follows from $S^N(k - 1)$ by merging the $i_k$ block into the $j_k$ block; that is $S_{i_k}(k) = \emptyset$ and 
$$
S_{j_k}^N(k) = S_{j_k}^N(k - 1) \cup S_{i_k}^N(k - 1).
$$
Equivalently at time $k$ we have $i_k$ meeting and losing to $j_k$.

From any initial configuration $S^N(0) = \gamma$, define $f_{\gamma} \colon \setN \rightarrow \setN$ by
$$
f_{\gamma}(k) = \inf S^N_k(0)
$$
for the owner of block $k$, noting the convention that if $S^N_k = \emptyset$ then $f_{\gamma}(k) = \infty$.

Call an initial configuration $\gamma$ \textbf{compatible} if for all meetings $1 \leq k \leq n$ we have
$$
f_{\gamma}(j_k) < f_{\gamma}(i_k),
$$
and say the initial configuration $\gamma$ \textbf{results} in $\alpha$ if 
$$
\{S^N(0) = \gamma \} \cap \bigcap_{k = 1}^n \{ k-\text{th meeting is between } i_k \text{ and }j_k \} \subset \{S^N(n) = \alpha \}.
$$
By our setup, $S^N(0)$ is uniform over all possible initial configurations - i.e. initial partitions $\gamma$ with $|\gamma| = |S^N(0)|$. Thus $\Prb ( S^N(n) = \alpha \vert F_n, \ldots, F_0)$ is the fraction of compatible initial configurations that result in $S^N(n) = \alpha$.

Instead of counting this directly, consider a fixed time $n$ partition $\alpha$. For any other partition $\beta$ with $|\beta| = |\alpha|$ we will simply show a bijection between compatible initial configurations resulting in each. Fix such a $\beta$.

Let $\tau$ be the unique bijection on $\{1, \ldots, N \}$ which for each $1 \leq i \leq N$ maps $\alpha_i$ to $\beta_i$ in an order preserving fashion; i.e. for each $1 \leq i \leq N$ the first element of $\alpha_i$ maps to the first of $\beta_i$, the second element of $\alpha_i$ maps to the second of $\beta_i$ etc. As we know $\# \alpha_i = \# \beta_i$ for each $i$ this is well defined.

If $\gamma$ results in $\alpha$, we claim the initial configuration $\tilde{\gamma}$ given by
$$
\tilde{\gamma}_i = \tau( \gamma_i)
$$ 
results in $\beta$. To check that $\tilde{\gamma}$ is compatible: after each meeting of blocks $i_k$ and $j_k$, $f_{\gamma}(i_k)$ and $f_{\gamma}(j_k)$ end up in the same block of the partition, and therefore at time $n$ are still in the same block $\alpha_i$ (for some $i$). Therefore as $f_{\gamma}(j_k) < f_{\gamma}(i_k)$ we also have $f_{\tilde{\gamma}}(j_k) < f_{\tilde{\gamma}}(i_k)$ since $\tau$ is order preserving within each block of $\alpha$.

Therefore, we find that for arbitrary $\alpha$, $\beta$ with $|\alpha| = |\beta| = |S^N|(n)$ that
$$
\Prb ( S^N(n) = \alpha \vert F_n, \ldots, F_0) = \Prb ( S^N(n) = \beta \vert F_n, \ldots, F_0).
$$
If $|\alpha| = |\beta| \neq |S^N|(n)$ then the equality is trivially true as both are zero.
\end{proof}
\end{prop}

From \Cref{prop:Uniform Partitions}, noting that the probability of any given permutation doesn't depend on the labelling $r_k$ of the blocks, \Cref{prop:Partition is Exchangeable} immediately follows.

\subsubsection{Asymmetric Exchangeability}

The second, asymmetric exchangeability property of the token process is exchangeability in the sense of de Finnetti's theorem. For it, we consider a fixed token $K$ and time $t \geq 0$ and look at the locations of the tokens $i \geq K + 1$. Recall that $u_i(t)$ is the owner of token $i$ at time $t$. The asymmetric exchangeability is given by the following.

\begin{prop}{For any fixed $K \geq 1$, $t \geq 0$, the sequence of indicators
$$
1( u_{K + 1}(t) = K), 1( u_{K + 2}(t) = K), 1(u_{K + 3}(t) = K), \ldots
$$
is, conditional on $\xi_1, \ldots, \xi_K$, an infinitely exchangeable sequence.
}
\label{prop:asymm exchg}
\end{prop}

Note that both \Cref{prop:asymm exchg} and \Cref{prop:Uniform Partitions} ultimately imply \Cref{prop:Partition is Exchangeable}, i.e that the token partition is exchangeable. The key difference is in where this exchangeability is coming from and what it is conditional on. The symmetric exchangeability is in relation to the history of the limiting process $\mu_t$, in relation to which (as we will see in \Cref{lem:Conditionally Exchg}) the different tokens behave symmetrically. The exchangeability here in  \Cref{prop:asymm exchg} is clearly asymmetric with respect to the tokens. The relationship between these two viewpoints is key in the sequel.

Our proof of \Cref{prop:asymm exchg} will make use of the combinatorial tool of \textbf{meeting trees} of the token process. For a fixed $N \geq 1$, we define a meeting tree $T$ of the first $N$ tokens as an ordered list
$$
T = (a_1, a_2, \ldots ; b_1, b_2, \ldots \vert i_1 \rightarrow j_1, \ldots, i_n \rightarrow j_n)
$$
which recounts the order of the meetings of the first $N$ tokens but not the (non-zero) times of these meetings.

We read the tree as follows. To the left of the divide are groups of tokens starting at the same initial location and thus meeting at time $t = 0$; i.e. for the generic tree $T$ above
$$
\xi_{a_1} = \xi_{a_2} = \ldots, \xi_{b_1} = \xi_{b_2} = \ldots,
$$
and so each such group meets at time $t = 0$ and merges into the lowest token of the group. We make no attempt at "ordering" meetings that occur simultaneously at time zero.

On the right of the divide, we write $a \rightarrow b$ for the meeting of tokens $a$ and $b$ resulting in $a$ merging into token $b$. We call a tree \textbf{valid} if
\begin{enumerate}
\item For each meeting $i_r \rightarrow j_r$, we have $1 \leq i_r < j_r \leq n$.
\item Each token $1$ to $N$ can merge into another token at most once, including meetings at $t = 0$.
\end{enumerate} 
Of course each token can be \textit{merged into} multiple times in the history of the process.

For any $t \geq 0$, write $S^N(t)$ for the (random) observed meeting tree at time $t$ of the $N$ token process. Then for any valid meeting tree $T$, the probability of $S^N(t) = T$ can be explicitly calculated as
$$
\Prb( S^N(t) = T \vert \xi_1, \ldots, \xi_N) = f_T(\xi_1, \ldots, \xi_N)
$$
for some function $f_T \colon S^N \rightarrow \R$ of the initial locations of the tokens. We'll see later (see \Cref{sec:second moment} and \Cref{sec:moment calc}) that for all but the smallest $N$ calculating $f_T$ explicitly is impractical.

\begin{figure}[h]
   \centering
\includegraphics[scale=0.3]{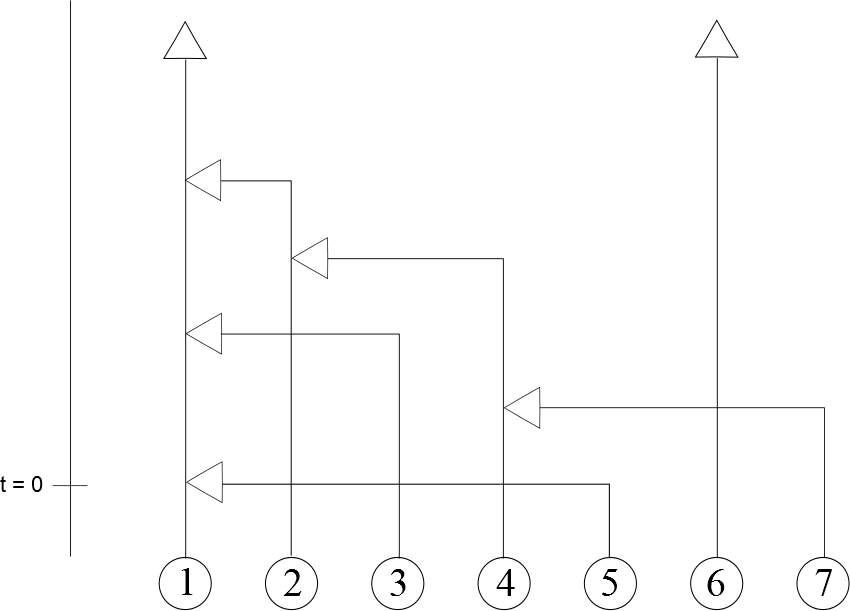}
   \caption{Meeting Tree $T = \left( 5, 1 \vert 7 \rightarrow 4, 3 \rightarrow 1, 4 \rightarrow 2, 2 \rightarrow 1 \right)$}
    \label{fig:tree}
\end{figure}

We are now ready to give a proof of \Cref{prop:asymm exchg}.

\begin{proof}
Fix $K \geq 1$ and $t \geq 0$. To show that the sequence
$$
1( u_{K + 1}(t) = K), 1( u_{K + 2}(t) = K), 1(u_{K + 3}(t) = K), \ldots
$$
is conditionally exchangeable, it suffices to show that for any $N > M \geq K +1$ that the conditional distribution of
$$
V = \Big( 1( u_{K + 1}(t) = K), \ldots 1(u_{M}(t) = K), 1(u_{M + 1}(t) = K), \ldots 1(u_{N}(t) = K) \Big)
$$ 
is the same as the conditional distribution of
$$
\hat{V} = \Big( 1( u_{K + 1}(t) = K), \ldots 1(u_{M + 1}(t) = K), 1(u_{M}(t) = K), \ldots 1(u_{N}(t) = K) \Big)
$$
i.e. under a transposition of $M$ and $M + 1$, since the set of transpositions generates the group of all finite permutations.

Consider any fixed binary sequences $s'$ and $s''$ of length $M - K$ and $N - (M - K + 2)$ respectively. Then for any binary sequence $s$ of length $N - K$ of the form
$$
s = (s', *, **, s'')
$$
where $*, ** \in \{0, 1 \}$ it suffices to show that
$$
\Prb \left( V = (s', *, **, s'') \vert \xi_i, 1 \leq i \leq K \right) = \Prb \left( \hat{V} = (s', *, **, s'')  \vert \xi_i, 1 \leq i \leq K \right)
$$
or equivalently
$$
\Prb \left( V = (s', *, **, s'') \vert \xi_i, 1 \leq i \leq K  \right) = \Prb \left( V = (s', **, *, s'') \vert \xi_i, 1 \leq i \leq K \right).
$$

Now, if $*$ and $**$ are both either $0$ or $1$ then this is trivially true, so we need only consider the case $* \neq **$. Without loss of generality assume $* = 0$ and $** = 1$. Let $\tau$ be the set of all meeting trees resulting in $V = (s', 0, 1, s'')$ and $\tau'$ the set of all meeting trees resulting in $V = (s', 1, 0, s'')$. Note that both $\tau$ and $\tau'$ are finite as there are only finitely many possible meeting trees on $N$ tokens.

We will construct a bijection $\psi \colon \tau \rightarrow \tau'$ with the property that for all $T \in \tau$
$$
\Prb( S^N(t) = T \vert \xi_i, 1 \leq i \leq K) = \Prb ( S^N(t) = \psi(T) \vert \xi_i, 1 \leq i \leq K ),
$$
which will show that
\begin{align*}
\Prb \Big( V = (s', 0, 1, s'') &\vert \xi_i, 1 \leq i \leq K  \Big) \\
&= \sum_{T \in \tau} \Prb \left( S^N(t) = T \vert \xi_i, 1 \leq i \leq K \right) \\
&= \sum_{T \in \tau}  \Prb \left( S^N(t) = \psi(T) \vert \xi_i, 1 \leq i \leq K \right) \\
&= \sum_{T \in \tau'} \Prb \left( S^N(t) = T \vert \xi_i, 1 \leq i \leq K \right) \\
&= \Prb \left( V = (s', 1, 0, s'') \vert \xi_i, 1 \leq i \leq K \right).
\end{align*}

For any $T \in \tau$, we define $\psi(T)$ to be the meeting tree given by switching the paths of the tokens $M$ and $M + 1$. A priori this need not be a valid meeting tree. In particular, if $M + 1 \rightarrow M$ occurs in $T$, then $\psi(T)$ would have the meeting $M \rightarrow M + 1$ and so wouldn't be a valid meeting tree. However, by assumption $T \in \tau$ has
$$
1( u_{M}(t) = K) = 0, \hspace{2 pc} 1( u_{M + 1}(t) = K) = 1
$$
and so at time $t$ token $M+1$ is owned by $K$ and token $M$ is not. Thus the paths of the tokens $M$ and $M + 1$ can not have already met in $T$, as otherwise either both would be owned by $K$ or neither would be. Therefore $\psi(T)$ is a valid meeting tree and so is clearly in $\tau'$. To see that $\psi$ is a bijection, we need only note that reversing the paths of $M$ and $M+1$ again gives the inverse $\psi^{-1} \colon \tau' \rightarrow \tau$.

To see $T$ and $\psi(T)$ have the same conditional probability we need only note that $\xi_M$ and $\xi_{M + 1}$ are exchangeable and
$$
f_{\psi(T)}(\xi_1, \ldots, \xi_N) = f_T (\xi_1, \ldots, \xi_{M + 1}, \xi_M, \ldots \xi_N)
$$
and so
\begin{align*}
\E \Big( f_{\psi(T)}(\xi_1, \ldots, \xi_N) &\vert \xi_i, 1 \leq i \leq K \Big) \\
&=\E \left( f_{T}(\xi_1, \ldots, \xi_{M +1}, \xi_M, \ldots, \xi_N) \vert \xi_i, 1 \leq i \leq K \right) \\
&= \E \left( f_{T}(\xi_1, \ldots, \xi_N) \vert \xi_i, 1 \leq i \leq K \right) 
\end{align*}
completing the proof.
\end{proof}

In fact our proof shows quite a bit more. Fixing a $K \geq 1$ and a time $t \geq 0$, for $n \geq K + 1$ consider the vector $V_n$ given by
$$
V_n = \left( 1(u_n(t) = 1), 1(u_n(t) = 2), \ldots, 1(u_n(t) = K) \right).
$$
It follows easily from our proof of \Cref{prop:asymm exchg} that conditional on the entire history of the first $K$ tokens up to time $t$ - i.e. on $\xi_1, \ldots, \xi_K$ and $Z^K(s), 0 \leq s \leq t$ that
$$
V_{K+ 1}, V_{K + 2}, V_{K + 3}, \ldots
$$
is an infinite exchangeable sequence. The value of \Cref{prop:asymm exchg} in the sequel is the ability to do computations by conditioning on only the initial conditions of finitely many tokens and so we content ourself with the current statement of the result.

\subsection{Proof of \Cref{prop:Token Limit}}

We are now ready to start our proof that the empirical measures $\mu^N_t$ coming from the token process converge weakly for each $t \geq 0$. For time $t = 0$, this is essentially the Glivenko-Cantelli theorem. 

\begin{lemma}{$\mu^N_0 \rightarrow \mu$ almost surely and thus $\mu_0 = \mu$ a.s.}
\end{lemma}

We note here however that while the convergence is easy, actually quantifying this convergence (in the Wasserstein metric, for instance) is hard and to our knowledge impossible in the absence of finite-dimensional type assumptions on $(S, d)$ \cite{boissard2011mean}. More classically for measures on $\R^d$ a bit more is known \cite{Horowitz1994261}. Interestingly, more can be said for the MC process at positive times $t > 0$, and we discuss later qualitative estimates of this convergence for positive times $t > 0 $ in \Cref{sec:TVD}.

We focus on showing weak convergence for times $t > 0$. Our main tool in this in Kingman's Paintbox Theorem, which serves as a structure theorem for exchangeable partitions.

\subsubsection{Kingman's Paintbox Theorem}

The key application of \Cref{prop:Partition is Exchangeable} is Kingman's Paintbox Theorem. Following Theorem 1.1 in \cite{berestycki2009recent}, a paintbox partition is derived from a mass partition $p = (p_0, p_1, \ldots)$ with $p_i \geq 0$ satisfying
$$
\sum_{i = 0}^{\infty} p_i = 1,
$$
and
$$
p_1 \geq p_2 \geq \ldots.
$$
Note that we do not require $p_0 \geq p_1$. Write $\mathfrak{P}_{\m}$ for the set of all such mass partitions.

Let $U_1, U_2, \ldots$ be an I.I.D. sequence of uniform $(0, 1)$ random variables and write
$$
I(u) = \inf \{ n \geq 0 \colon \sum_{m = 0}^n p_i > u \}.
$$
Intuitively, if $p$ is a partition of the unit interval, then $I(u)$ is the block of partition that $u$ falls into. Then, the paintbox partition given by $p$ is defined by $i \sim j$ if and only if
$$
I(U_i) = I(U_j) \geq 1,
$$
i.e. if $U_i$ and $U_j$ fall in the same (non $p_0$) "slot" of the paintbox. The mass $p_0$ corresponds to singletons blocks of the partition, called \textbf{dust}. An elementary fact of exchangeable partitions is that all blocks must be either singletons or infinite, and don't take on any other finite sizes.

Clearly this construction gives an exchangeable partition. The importance of Kingman's Theorem is that this is a universal construction for such random partitions. Our statement of the theorem follows Theorem 1.1 in \cite{berestycki2009recent}.

\begin{theorem}{Kingman's Paintbox: Let $\pi$ be any exchangeable partition. Then there exists a probability measure $\alpha(dp)$ on $\mathfrak{P}_{\m}$ such that
$$
\Prb(\pi \in \cdot ) = \int_{p \in \mathfrak{P}_{\m}} \alpha(dp) \rho_p( \cdot ),
$$
where for a mass partition $p$, $\rho_p$ is the corresponding law on $\Pi$ given by the paintbox construction.
}
\label{thrm:Kingman Paintbox}
\end{theorem}

\subsubsection{Defining $\mu_t$}

Applying Kingman's correspondence, the token process' partition $Z(t)$ can be described as follows. First, there exists a (random) mass partition $p(t) = \{ p_k(t), k \geq 0 \}$ with
\begin{align}
\sum_{k = 0}^{\infty} p_k(t) =1, \label{eq:Sum of pk}
\end{align}
giving a random partition
$$
B(t) = \{B_k(t), k \geq 1 \},
$$
of $\N$, which will have asymptotic rates $p_k(t)$. Note that the partition $B(t)$ will be the same as $Z(t)$, however is no longer indexed by the lowest token in each block.

Then, the equivalence relation generating the token partition $\sim_t$ is given by
$$
i \sim_t j
$$
if and only if $i$ and $j$ are in the same block $B_k(t)$ for $k \geq 1$.

Conditional on $p(t)$, each token $i \geq 1$ is assigned a block of $B(t)$ independently according to $p(t)$. That is, token $i$ is assigned to block $B_i$ with probability $p_i(t)$ for $i \geq 1$. The remaining mass $p_0(t)$ corresponds to token $i$ being a singleton (or dust), which in our setting would be tokens that have had no meetings by time $t$. There are many different ways to show that this doesn't happen almost surely; our proof uses the asymmetric exchangeability of \Cref{prop:asymm exchg}.

\begin{lemma}{For any $t > 0$, almost surely $p_0(t) = 0$.}
\label{lem:No Dust}
\begin{proof}
By construction, $\E p_0(t)$ is the probability that any given token is a singleton at time $t$, in particular we'll consider the probability that token $1$ is a singleton at time $t$. The key here is to calculate this using the second form of exchangeability.

By \Cref{prop:asymm exchg} the sequence
$$
1(u_2(t) = 1), 1(u_3(t) = 1 ), \ldots
$$
is conditionally exchangeable given $\xi_1$. Applying de Finetti's theorem to this sequence of Bernoulli variables (Theorem 4.6.6, \cite{durrett2010probability}) we see that there is a $q_1(t) \in \sigma(\xi_1)$ satisfying $0 \leq q_1(t) \leq 1$ such that conditional on $q_1(t)$, the sequence is independent and identically distributed (as $\Ber(q_1(t))$ random variables).

On the other hand, by the definition of the token process we have that
\begin{align*}
\Prb( u_2(t) = 1 \vert \xi_1) &= \Prb( t_{12} \leq t \vert \xi_1) \\
&= \E \left( 1 - \exp( - \phi(d(\xi_1, \xi_2))t) \vert \xi_1 \right) \\
&> 0,
\end{align*}
almost surely for all $t > 0$. Therefore, we have that
$$
q_1(t) > 0
$$
almost surely.

Since $1( u_i(t) = 1)$ are independent conditional on $q_1(t)$, by the Strong Law they can only all be zero if $q_1(t)$ is itself zero. Thus, we can conclude that
\begin{align*}
\Prb( 1 \text{ is dust at } t) &= \Prb \left( \forall i \geq 2, 1(u_i(t) = 1) = 0 \right) \\
&=\E \Prb \left( \forall i \geq 2, 1(u_i(t) = 1) = 0 \vert \xi_1 \right) \\
&=\E \Prb( q_1(t) = 0 \vert \xi_1 ) \\
&= 0
\end{align*}
whenever $t > 0$, completing the proof.
\end{proof}
\end{lemma}

We are now finally able to define the random measure $\mu_t$. By construction of the token process, two tokens $i$ and $J$ have $i \sim_t j$ if and only if they are at the same location of $S$ at time $t$. By \Cref{lem:No Dust}, this gives us a one-to-one correspondence between blocks $B_k(t)$ and locations $s_k \in S$. Namely, writing
$$
l_k(t) = \inf B_k(t)
$$
for the lowest token in block $B_k$, we have
$$
s_k(t) = \xi_{l_k(t)}.
$$

Therefore we can define our random measure $\mu_t$ by
\begin{align}
\mu_t = \sum_{k = 1}^{\infty} p_k(t) \delta(s_k(t)). \label{def:mut}
\end{align}
By \Cref{lem:No Dust} and \Cref{eq:Sum of pk} we are guaranteed that $\mu_t$ is in fact a probability measure on $S$.

All that remains of \Cref{prop:Token Limit} is to show that for all $t > 0$
$$
\mu^N_t \rightarrow \mu_t
$$
almost surely.

\subsubsection{Convergence}

Later (\Cref{sec:TVD}) we give of a proof of the convergence of the empirical measures using Total Variation distance. However, we'll need to make use later of the concept of a convergence determining class for weak convergence and so we opt here to give such a proof of \Cref{prop:Token Limit}.

For any $k \geq 1$, write
$$
p^N_k(t) = \frac{1}{N} \sum_{n = k}^{N} 1(n \in B_k(t))
$$
for the empirical mass at $s_k$ of the first $N$ tokens. We first show that for any function $f \in C_b(S)$, the measures $\mu_t^N$ converge to $\mu_t$ almost surely.

\begin{lemma}{For any fixed $f \in C_b(S)$ and for all $t > 0$, 
$$
\mu^N_t(f) \rightarrow \mu_t(f)
$$
 almost surely.}
 \label{lem:token limit for f}
\begin{proof}
For $i \geq 1$ write $f_i, i \geq 1$ for the sequence
$$
f_i(t) = \sum_{i = k}^{\infty} f(s_k(t)) 1( i \in B_k(t)).
$$
By assumption $f_i(t)$ is bounded, since $f$ is and only one such indicator occurs. Clearly
$$
\mu_t^N(f) = \frac{1}{N}\sum_{i = 1}^N f_i(t).
$$
Recalling the Paintbox construction, conditional on $p(t)$ the sequence $f_i$ is independent and so applying the Strong Law we can conclude that
$$
\mu_t^N(f) \rightarrow \mu_t
$$
almost surely.
\end{proof}
\end{lemma}

To complete the proof of \Cref{prop:Token Limit} we need to show that $\mu_t^N(f) \rightarrow \mu_t(f)$ for all $f \in C_b(S)$, almost surely. To show this from \Cref{lem:token limit for f}, we use a standard reduction taking advantage of the fact that $S$ is separable.

The key is the standard result that there exists a countable class of test functions for weak convergence, called a \textbf{convergence determining class} \cite{billingsley2009convergence}. This is opposed to needing to test all of $C_b(S)$ for convergence, which is (for generic $S$) uncountable. Since we'll apply this same reduction a few times in the sequel, we expand on the details here.


\begin{theorem}{If $S$ is a separable metrizable space, then there exists a countable subset $\mathfrak{C} \subset C_b(S)$ such that for any sequence of measures $\nu^i, 1 \leq i \leq \infty$, $\nu^i \rightarrow \nu^{\infty}$ weakly if and only if $\nu^i(f) \rightarrow \nu^{\infty}(f)$ for all $f \in \mathfrak{C}$.}
\label{thrm:Test Class}
\end{theorem}

An immediate application to our setting is the following corollary.

\begin{prop}{If $S$ is a separable metrizable space, then for any sequence of random measures $\nu^i \in P(S), 1 \leq i \leq \infty$ defined on the same probability space, $\nu^i \rightarrow \nu^{\infty}$ in the weak topology almost surely if and only if for every fixed $f \in C_b(S)$, $\nu^i(f) \rightarrow \nu^{\infty}(f)$ almost surely.}
\label{thrm:Random Measure Convergence}
\end{prop}

Applying \Cref{thrm:Random Measure Convergence} and \Cref{lem:token limit for f} the proof of \Cref{prop:Token Limit} us complete.

\subsection{Support and Total Variation Distance}
\label{sec:TVD}

We conclude this section with a brief discussion of the support of the measures $\mu_t, t \geq 0$ and give an easy corollary of \Cref{prop:Token Limit}.

Note that the measures $\mu^N_t$ have finitely many atoms and that their support coincides with the locations at time $t$ of the 
first $N$ tokens, that is
$$
\supp \mu^N_t = \{ \xi_{u_i(t)} \colon 1 \leq i \leq N \}.
$$
For $\mu_t$, by an abuse of notation for now we'll define the support for $t > 0$ as the similarly defined set of token's 
locations, i.e. we'll define $\supp \mu_t$ by
\begin{align}
\supp \mu_t = \{ \xi_{u_i(t)} \colon i \geq 1 \}. \label{def:supp mut}
\end{align}
We'll show later (see \Cref{sec:Finite Support}) that for positive times the set of token locations is finite on every compact 
set of $S$, and so is nowhere dense. This implies that our definition is in fact equivalent to the standard definition of the 
support of the measure, i.e
$$
\supp \nu = \overline{\{ s \in S \colon \nu(U) > 0 \text{ for all open } U \ni s \} },
$$
as shown in \Cref{lem:Support matches standard}. In particular, this shows that for all positive times $t$, the support of $\mu_t$ is countable. An immediate and useful corollary 
of this is the following.

\begin{lemma}{For all $t > 0$
$$
\supp \mu^N_t \subset \supp \mu_t.
$$}
\label{lem:Support of mu}
\end{lemma}

We can now prove the following strengthening of \Cref{prop:Token Limit} which we will need later, showing that $\mu^N_t$ converges to $\mu_t$ in the stronger topology of Total Variation distance. We recall the definition of the Total Variation distance for a countable state space\cite{gibbs2002choosing}: for two measures $m_1$ and $m_2$ on a countable state space $X$, the total variation distance $d_{\TV}$ is given by
$$
d_{\TV}(m_1, m_2) = \frac{1}{2} \sum_{x \in X}| m_1(x) - m_2(x) |.
$$
As the token measures $\mu_t^N$ and their limit $\mu_t$ share a countable state space, by a fairly standard argument we can calculate their Total Variation distance.

\begin{lemma}{For any $t > 0$, we have
$$
d_{\TV}(\mu^N_t, \mu_t) \rightarrow 0
$$
almost surely.}
\label{lem:Convergence in TVD}
\begin{proof}
Fix any $\epsilon > 0$. Then, because
$$
\sum_{k = 1}^{\infty} p_k(t) = 1
$$
almost surely, we can almost surely choose a $K$ so that
\begin{align}
\label{eq:sum of pk}
\sum_{k = K + 1}^{\infty} p_k(t) < \epsilon.
\end{align}
Next, the Strong Law applied to the paintbox construction gives that $p_k^N(t) \rightarrow p_k(t)$ for all $ k \geq 1$ almost surely. Thus, we can choose almost surely an $M$ such that if $N \geq M$ then
$$
\sum_{k = 1}^K |p^N_k(t) - p_k(t)| \leq \epsilon.
$$

We claim that
$$
\sum_{k = K + 1}^{\infty} p^N_k(t) < 2 \epsilon.
$$
To see this, we calculate
\begin{align*}
| \sum_{K + 1}^{\infty} p_k^N(t) - \sum_{K + 1}^{\infty} p_k(t)| 
&= |\left( 1 - \sum_{1}^K p_k^N(t)\right) - \left( 1 - \sum_{1}^K p_k(t)\right) | \\
&= |\sum_{1}^K p_k^N(t) - \sum_1^K p_k(t) | \\
&\leq \sum_{k = 1}^K |p^N_k(t) - p_k(t)| \\
&\leq \epsilon,
\end{align*}
which combined with \Cref{eq:sum of pk} proves the claim.

Thus, for $N \geq M$, we have that
\begin{align*}
d_{\TV}(\mu^N_t, \mu_t) &= \frac{1}{2} \sum_{k = 1}^{\infty}|p^N_k(t) - p_k(t)| \\
&\leq \frac{1}{2} \sum_{k = 1}^{K}|p^N_k(t) - p_k(t)| +  \frac{1}{2}\sum_{K + 1}^{\infty} \left( p_k^N(t) + p_k(t) \right) \\
&\leq \frac{5}{2} \epsilon,
\end{align*}
completing the proof.
\end{proof}
\end{lemma}

Note that \Cref{lem:Convergence in TVD} is trivially false for time $t = 0$. In particular, if $\mu$ is non-atomic, as all of the empirical measures $\mu_0^N$ are atomic, they can never converge to $\mu_0 = \mu$ in Total Variation distance.

Note that in light of \Cref{lem:Convergence in TVD} it seems natural to wonder if $\mu^N_t$ converges to $\mu_t$ uniformly in Total Variation distance for $t > 0$. We answer this question in the affirmative for $t$ bounded away from $0$ in \Cref{lem:uniform tvd}.

\section{The Markov Property}
\label{sec:Markov Property}

In this section, we'll prove that from any initial measure $\mu \in P(S)$, that the constructed process $\mu_t, t \geq 0$ satisfies the Markov Property.

\begin{prop}{From any initial $\mu$, the process $\mu_t, t \geq 0$ is a time-homogeneous Markov process.}
\label{prop:mu Markov}
\end{prop}

Our approach to this will be to show first that the empirical measures $\mu_t^N, t \geq 0$ of the token process are Markovian with respect to the adapted filtration of $\mu_t, t \geq 0$. Then, by a limiting argument we will show that this carries over for $\mu_t$. Finally, we use the exchangeability of the token process to prove time homogeneity.

\subsection{Markov Condition for $\mu_t^N$}

To begin, we show that the empirical processes $\mu_t^N, t \geq 0$ are Markov with respect to the filtration generated by the MC process $\mu_t, t \geq 0$.

\begin{prop}{For any Borel $A \subset P(S)$, times $0 \leq t_0 \leq t_1$ and $N \geq 1$ we have
$$
\Prb ( \mu^N_{t_1} \in A \vert F_{t_0}) = \Prb ( \mu^N_{t_1} \in A \vert \mu_{t_0} ).
$$}
\label{lem:muN Markov}
\end{prop}

To prove \Cref{lem:muN Markov}, we fix a time $t_0$ and make use of the symmetric paintbox description of the measure $\mu_{t_0}$ as
$$
\mu_{t_0} = \sum_{k = 1}^{\infty} p_k(t_0) \delta(s_k).
$$

Now each token $1\leq i \leq N$ at time $t_0$ is in one of the blocks of the paintbox partition given by $p(t_0)$, so write $A_i$ for the block of the paintbox partition containing token $i$. That is, the location in $S$ of token $i$ at $t_0$ is
$$
\xi_{u_i(t_0)} = s_{A_i}.
$$
For simplicity of notation we will often write
$$
A^N = (A_1, \ldots, A_N),
$$
for the location of the first $N$ tokens at time $t_0$. By \Cref{prop:Uniform Partitions}, we already know that at time $t_0$ the token partition $Z(t_0)$ is exchangeable. In fact, we will show that even conditional on $\mu_t, 0 \leq t \leq t_0$, the token process is still exchangeable. This is the complete description of the symmetric form of exchangeability for the MC process (as in \Cref{sec:Symmetric Exchg}).

\begin{prop}{Conditional on $\mu_t, 0 \leq t \leq t_0$ the token partition $Z(t_0)$ is exchangeable.}
\label{lem:Conditionally Exchg}
\begin{proof}
\Cref{prop:Uniform Partitions} shows that conditional on $\mu^N_t, 0 \leq t \leq t_0$ the partition $Z^N(t_0)$ is exchangeable. To complete this proof, we need to show that this exchangeability holds conditional on $\mu_t, t \geq 0$.

A (standard) basis for the topology of the set $\Pi$ of all permutations on $\N$ is given by the finite restrictions, that is by the sets of the form
$$
B_{\pi_0} = \{\pi \in \Pi \colon \pi|_N = \pi_0|_N \}
$$ 
where $\pi|_N = \pi \cap \{1, 2 \ldots, N \}$. Consider any such fixed $B = B_{\pi_0}$.

Let $\sigma$ be a finite permutation of $\N$. By assumption there is an $N_{\sigma} \geq 1$ such that $\sigma$ is constant for all integers after $N_{\sigma}$. Now, \Cref{prop:Uniform Partitions} implies that the conditional law of the partition $Z^N(t_0)$ on the set $\Pi$ of all positive partitions satisfies
$$
\Prb( \sigma Z^N(t_0) \in B \vert \mu^N_t, 0 \leq t \leq t_0) = \Prb( Z^N(t_0) \in B \vert \mu^N_t, 0 \leq t \leq t_0),
$$
for all $N \geq N_{\sigma}$.

To show that $Z(t_0)$ is conditionally exchangeable (on $B$), we need to show that for any such $\sigma$,
\begin{align}
\Prb( \sigma Z(t_0) \in B \vert \mu_t, 0 \leq t \leq t_0) = \Prb( Z(t_0) \in B \vert \mu_t, 0 \leq t \leq t_0). \label{eq:sigmaZ cond}
\end{align}

Now, by the finite definition of $B$, we have
$$
1( Z^N(t_0) \in B ) \rightarrow 1( Z(t_0) \in B)
$$
and similarly for $\sigma Z^N(t_0)$, as for $N \geq N_{\sigma}$ we have
$$
(\sigma Z(t_0) )|_N = \sigma Z^N(t_0).
$$

We can then prove \Cref{eq:sigmaZ cond} by a standard limiting and monotone class argument. First, for any finite set of times $t_1, \ldots, t_n$ we have $\mu^N_{t_i} \rightarrow \mu_{t_i}$ for all $1 \leq i \leq n$ almost surely. Letting $f_i \colon P(S) \rightarrow \R$ be bounded and continuous, we then have
$$
F^N = f_1(\mu^N_{t_1}) \cdots f_n(\mu^N_{t_n}) \rightarrow F = f_1(\mu_{t_1}) \cdots f_n(\mu_{t_n}),
$$
almost surely. Now by \Cref{prop:Uniform Partitions} for $N \geq N_{\sigma}$ we have that
$$
\E \left( 1(\sigma Z^N(t_0) \in B ) F^N \right) = \E \left( 1( Z^N(t_0) \in B ) F^N \right).
$$

Applying the bounded convergence theorem, this shows that for any
$$
F \in \sigma( \mu_t, 0 \leq t \leq t_0)
$$
that can be factored as such, we have
\begin{align*}
\E \left( 1(\sigma Z(t_0) \in B ) F \right) &= \E \lim_N \left( 1(\sigma Z^N(t_0) \in B )F^N \right)\\
&= \lim_N \E \left( 1\left( \sigma Z^N(t_0) \in B \right) F^N \right ) \\
&= \lim_N \E \left( 1\left( Z^N(t_0) \in B \right) F^N \right ) \\
&=  \E \left( 1(Z(t_0) \in B ) F \right).
\end{align*}
By the Monotone Class Theorem (Theorem 5.1.5, \cite{durrett2010probability}), the same equality holds for any bounded $F \in \sigma( \mu_t, 0 \leq t \leq t_0)$, proving \Cref{eq:sigmaZ cond}. As the finitely defined sets $B$ are a basis for the topology on $\Pi$, the $\pi-\lambda$ theorem shows that \Cref{eq:sigmaZ cond} holds for any measurable subset of $\Pi$, completing the proof.
\end{proof}
\end{prop}

Having proved \Cref{lem:Conditionally Exchg}, we can now easily see that conditional on the past of the MC process, the partition at time $t_0$ is still given by the paintbox construction.

\begin{lemma}{For any $k_1, k_2, \ldots, k_N \in \{1, 2 \ldots \}$, we have that
$$
\Prb( A^N = (k_1, \ldots, k_N) \vert \mu_t, 0 \leq t \leq t_0) = \Prb( A^N = (k_1, \ldots, k_N) \vert \mu_{t_0}).
$$}
\label{lem:Cond A 1}
\begin{proof}
\Cref{lem:Conditionally Exchg} shows that conditional on $\mu_t, 0 \leq t \leq t_0$, the token partition $Z(t_0)$ is exchangeable and so given by a paintbox construction, which since the asymptotic rates (i.e mass partition) at time $t_0$ are determined by $\mu_{t_0}$ must coincide with the mass partition in the construction of $\mu_t$. Therefore we have
$$
\Prb( A^N = (k_1, \ldots, k_N) \vert \mu_t, 0 \leq t \leq t_0) = p_{k_1}(t_0) \cdots p_{k_n}(t_0)
$$
which is clearly in $\sigma(\mu_{t_0})$.
\end{proof}
\end{lemma}

Next, we show that conditional on the past of the MC process, the future is determined solely by $\mu_{t_0}$ and the locations $A^N$ of the first $N$ tokens. 

\begin{lemma}{For $0 \leq t_0 \leq t_1$ we have that
$$
\Prb(\mu^N_{t_1} \in \cdot \vert \mu_t, 0 \leq t \leq t_0, A^N) = \Prb(\mu^N_{t_1} \in \cdot \vert \mu_{t_0}, A^N),
$$
as random measures on $P(S)$.}
\begin{proof}
The basic idea of this proof is that the locations of the first $N$ tokens at time $t_1$ conditional on the past (before $t_0$) depends only on the locations of the first $N$ tokens at time $t_0$.

For shorthand write $\xi_i(t)$ for the location in $S$ of token $i$ at time $t$, i.e.
$$
\xi_i(t) = \xi_{u_i(t)} = s_{A_i},
$$
and let $\xi^N(t)$ be the collection $\xi_i(t), 1 \leq i \leq N$ and similarly $\xi(t)$ for $\xi_i(t), i \geq 1$.

First, by construction the token process $\xi(t), t \geq 0$ is Markov and so as $\mu^N_t$ is a function of $\xi(t)$ we have
\begin{align*}
\Prb( \mu^N_{t_1} \in \cdot \vert \xi(t), 0 \leq t \leq t_0 ) &= \Prb( \mu^N_{t_1} \in \cdot \vert \xi(t_0) )\\
&= \Prb( \mu^N_{t_1} \in \cdot \vert \xi^N(t_0) )
\end{align*}
using for the second equality that the path of the first $N$ tokens doesn't depend on the path of the other tokens.

Therefore using that
$$
\sigma ( \mu_t, 0 \leq t \leq t_0, A^N) \subset \sigma( \xi(t), 0 \leq t \leq t_0 )
$$
we can complete the proof using the tower property of conditional expectation by calculating
\begin{align*}
\Prb( \mu^N_{t_1} \in \cdot &\vert  \mu_t, 0 \leq t \leq t_0, A^N) \\
&=\E \left( \Prb \left( \mu^N_{t_1} \in \cdot \vert \xi(t), 0 \leq t \leq t_0 \right) \vert \mu_t, 0 \leq t \leq t_0, A^N \right) \\
&= \E \left( \Prb \left( \mu^N_{t_1} \in \cdot \vert \xi^N(t_0) \right) \vert \mu_t, 0 \leq t \leq t_0, A^N \right) \\
&= \Prb( \mu^N_{t_1} \in \cdot \vert \mu^N_{t_0}, A).
\end{align*}
For the last step, we use that
$$
\sigma ( \xi^N(t_0) ) = \sigma(\mu^N_{t_0}, A^N ) \subset \sigma(  \mu_t, 0 \leq t \leq t_0, A^N ).
$$
\end{proof}
\label{lem:Cond A 2}
\end{lemma}

We are now able to give our proof of \Cref{lem:muN Markov}.

\begin{proof}
Applying  \Cref{lem:Cond A 1} and \Cref{lem:Cond A 2}, by a simple calculation we have
\begin{align*}
\Prb(\mu^N_{t_1} \in \cdot &\vert \mu_t, 0 \leq t \leq t_0) \\
&= \sum_{A^N = \bar{k}} \Prb(\mu^N_{t_1} \in \cdot, A^N = \bar{k} \vert \mu_t, 0 \leq t \leq t_0 ) \\
&= \sum_{A^N = \bar{k}} \Prb(\mu^N_{t_1} \in \cdot \vert \mu_t, 0 \leq t \leq t_0, A^N) \Prb(A^N = \bar{k} \vert \mu_t, 0 \leq t \leq t_0 ) \\
&= \sum_{A^N = \bar{k}} \Prb(\mu^N_{t_1} \in \cdot \vert \mu_{t_0}, A^N) \Prb(A^N = \bar{k} \vert \mu_{t_0} ) \\
&= \Prb(\mu^N_{t_1} \in \cdot \vert \mu_{t_1})
\end{align*}
where the sum is over all $\bar{k} = (k_1, \ldots, k_N) \subset \{1, 2, \ldots \}^N$.
\end{proof}

\subsection{Proof of \Cref{prop:mu Markov}}

We can now start the proof of \Cref{prop:mu Markov} and show that $\mu_t, t \geq 0$ is Markov, by applying \Cref{lem:muN Markov} and a limiting argument.

\begin{proof}
First, we claim that as random measures on $P(S)$ we have
$$
\Prb(\mu^N_{t_1} \in \cdot \vert \mu_t, 0 \leq t \leq t_0) \rightarrow \Prb(\mu_{t_1} \in \cdot \vert \mu_{t}, 0 \leq t \leq t_0 ),
$$
almost surely.

To see this, let $U \subset P(S)$ be open. Then as $\mu^N_{t_1} \rightarrow \mu_{t_1}$, if $\mu_{t_1} \in U$ then almost surely so is $\mu^N_{t_1}$ eventually; that is
$$
\liminf_N 1(\mu^N_{t_1} \in U) \geq 1(\mu_{t_1} \in U).
$$
Applying Fatou's Lemma we have that
\begin{align*}
\liminf_N \E (1 (\mu^N_{t_1} \in U) \vert \mu_t, 0 \leq t \leq t_0 ) &\geq \E( \liminf_N 1 (\mu^N_{t_1} \in U) \vert \mu_t, 0 \leq t \leq t_0 )  \\
&\geq \E(1(\mu_{t_1} \in U) \vert \mu_t, 0 \leq t \leq t_0),
\end{align*}
that is
$$
\liminf_N \Prb(\mu^N_{t_1} \in U \vert \mu_t, 0 \leq t \leq t_0) \geq \Prb( \mu_{t_1} \in U \vert \mu_t, 0 \leq t \leq t_0)
$$
almost surely. Applying the Portmanteau theorem to $P(S)$, we have proved the claim.

By the same argument, we also have
$$
\Prb(\mu^N_{t_1} \in \cdot \vert \mu_{t_0}) \rightarrow \Prb(\mu_{t_1} \in \cdot \vert \mu_{t_0} ),
$$
almost surely as measures. By \Cref{lem:muN Markov}, we have that
$$
\Prb(\mu^N_{t_1} \in \cdot \vert \mu_{t_0}) = \Prb(\mu^N_{t_1} \in \cdot \vert \mu_{t}, 0 \leq t \leq t_0)
$$
almost surely, and so we can conclude that almost surely
$$
 \Prb(\mu_{t_1} \in \cdot \vert \mu_{t_0} ) =  \Prb(\mu_{t_1} \in \cdot \vert \mu_{t}, 0 \leq t \leq t_0)
$$
completing the proof.
\end{proof}

All that is left to show to complete the proof of \Cref{prop:mu Markov} is to prove that the MC process $\mu_t, t \geq 0$ is time-homogeneous. Write $\Prb_{\mu}$ for the distribution of the token process started at $\mu$. Time homogeneity will follow immediately from the following lemma.

\begin{lemma}{For any time $0 \leq t_0$, the distribution of the token process conditional on $\mu_t, 0 \leq t \leq t_0$ is the same as that of the token process started from $\mu_{t_0}$. That is
$$
\Prb_{\mu}( Z(t_0) \in \cdot \vert \mu_t, 0 \leq t \leq t_0) = \Prb_{\mu_{t_0}}( Z(0) \in \cdot ).
$$}
\label{lem:Time homog}
\begin{proof}
This follows easily from \Cref{lem:Conditionally Exchg}, which shows that conditional on $\mu_t, 0 \leq t \leq t_0$, the location $A_i \in S$ of each token $i \geq 1$ is independent and given by the paintbox $p(t_0)$. It is immediate then that this is the same as the initial distribution of the token process for the countably supported measure $\mu_{t_0}$, since the token process is Markov and both have the same initial distribution of token locations.
\end{proof}
\end{lemma}

\section{Finite Support and Kingman's Coalescent}
\label{sec:Finite Support}

In this section, we begin our comparison between the token process construction and the Metric Coalescent, which we recall is originally only defined on $P_{\fs}(S)$. Our main result here shows that on any compact set in $S$, the support of the token process $\mu_t, t \geq 0$ 'comes down from infinity'.

\begin{prop}{For any initial measure $\mu \in P(S)$ and any compact set $C \subset S$
$$
\#  \left( \supp \mu_t \cap C \right) < \infty
$$
for all $t > 0$, almost surely.}
\label{prop:K finite}
\end{prop}

In fact, for any compact $C$ and time $t > 0$, we can also show that the expected number of atoms in $C$ at time $t$ is finite. 

\begin{prop}{For any initial measure $\mu \in P(S)$, any compact set $C \subset S$, and any $t > 0$ we have
$$
\E \left( \supp \mu_t \cap C \right) \leq \frac{2}{t \phi_{\min}(C)}.
$$
}
\label{prop:K finite exp}
\end{prop}

Here we define $\phi_{\min}(C)$ to be the positive infinum of $\phi$ on $C$, that is
$$
\phi_{\min}(C) = \inf_{x,y \in C} \phi(d(x, y)) > 0,
$$
which by compactness of $C$ is immediately bounded away from zero.

An easy consequence of \Cref{prop:K finite} is our claim from \Cref{sec:TVD} that the support of $\mu_t$ - as we have defined it, given by the set of token locations at time $t$ - is nowhere dense and so matches the standard definition of support.

\begin{lemma}{For all $t > 0$, the support of $\mu_t$ is nowhere dense and so almost surely
$$
\supp \mu_t = \overline{\{ s \in S \colon \forall \mu_t(U) > 0 \text{ for all open } U \ni s \} }.
$$}
\label{lem:Support matches standard}
\end{lemma}

Also trivial, by the definition of the token process and \Cref{lem:Support matches standard}, is that the support of $\mu_t$ is almost surely contained in that of the initial measure $\mu$. This gives as an immediate corollary of \Cref{prop:K finite} a portion of \Cref{thrm:Main Theorem}, namely that for any compactly supported initial measure $\mu \in P_{\cs}(S)$ that $\mu_t, t \geq 0$ comes down from infinity.

\begin{prop}{If $\mu \in P_{\cs}(S)$, then
$$
\# \supp \mu_t < \infty
$$
for all $ t > 0$, almost surely.}
\label{prop:Compact Support Comes Down}
\end{prop}

Note that while \Cref{prop:Compact Support Comes Down} shows that our extension of the MC comes down from infinity on $P_{\cs}(S)$, this doesn't exclude it from doing so on a possibly larger class of measures in $P(S)$. We'll discuss this possibility in \Cref{sec:Open Probs}.

We begin with some bounds on $\# \left( \supp \mu^N_t \cap C \right)$ for $t > 0$. To do this, we compare the token process to the well studied Kingman Coalescent (KC), a classical coalescing partition process.

\subsection{Kingman's Coalescent}
\label{sec:Kingman}

Kingman's Coalescent is a classical coalescent process on the natural numbers $\N$ with applications throughout probability theory, in particular for population models in genetics where it satisfies an important universality property known as M\"{o}hle's Lemma. For our purposes, the connections between Kingman's Coalescent and the Metric Coalescent are deep, with the Metric Coalescent appearing as essentially a version of Kingman's model with dependent (exchangeable) rates.

Kingman's Coalescent is simple to describe, although a construction of it as a limit of finite processes takes some care. Simply put, it is a coalescing partition valued Markov process on $\N$, where each pair of blocks merges independently with some fixed rate $\gamma > 0$. Compare this to the MC, where to the token process, where the rates of merger between blocks are of course geometry dependent.

A standard fact about KC is that it \textit{comes down from infinity}, that is, started from the initial partition of all singletons, for any positive time $t > 0$ the number of blocks $N^{KC}_t < \infty$ almost surely; in fact a stronger statement, bounding $\E N^{KC}_t$ holds as well \cite{berestycki2009recent}.

While the rates of the token process are not independent, for any two atoms $s_i, s_j \in C$ of $\mu_t$ they satisfy 
$$
\nu_{ij} \geq \phi_{\min}(C) > 0,
$$
for any compact set $C$. Thus by comparison to the KC we will show that the token process also comes down from infinity.

\subsection{Coming Down From Infinity}

While a direct comparison between the two process through a coupling may be possible, we find it simpler to apply the classical proof of finiteness for KC to the token process, with few adjustments. Our proof here follows \cite{berestycki2009recent} closely.

Fix some compact set $C \subset S$ and define
$$
K_t(C) = \sum_{k = 1}^{\infty} 1(u_k(t) = k, \xi_k \in C), \label{def:Define K}
$$
that is, $K_t$ is the (possibly infinite) number of tokens that haven't merged by time $t$ in $C$. Similarly we define
$$
K^N_t = \sum_{k = 1}^{N} 1( u_k(t) = k, \xi_k \in C), \label{def:Define KN}
$$
for the $N$ token process. Note that by our construction
$$
K^N_t(C) = \# \left( \supp \mu^N_t \cap C \right)
$$
as at most one alive token can occupy each point of $S$. Our main result in this section is the following.

\begin{prop}{For all $t > 0$, almost surely $K_t(C) < \infty$.}
\label{prop:K finite for t}
\end{prop}

We will follow the classical proof for Kingman's Coalescent to prove \Cref{prop:K finite for t}. As indicated, the exact same proof method borrowed from Kingman's Coalescent in fact yields the expectation bound in \Cref{prop:K finite exp} for $K_t(C)$, (see Theorem $2.1$ in \cite{berestycki2009recent}) whose details we omit here.

Note for \Cref{prop:K finite exp} what will be important for our purposes is not the precise bound on $\E K_t(C)$ but only that such a bound exits for all $t > 0$. The rest this section will be outlining the proof of \Cref{prop:K finite for t}.

Consider a fixed initial measure $\mu \in P(S)$. For $1 \leq n \leq N$, let
$$
T^N_n = \inf \{ t \geq 0 \colon K^N_t(C) = n \}.
$$
The key to our proof is that while the pure death process $K^N_t(C)$ is typically not Markovian, the token process is and so when there are $K^N_t(C) = n$ blocks left, the conditional intensity of a transition from $n$ to $n - 1$ is at least $\binom{n}{2} \phi_{\min}(\mu)$, the rate at which two atoms of $\mu_t^N$ within $C$ merge. In fact, it is a priori much higher than this, since an atom from $C$ can merge into an atom outside of $C$.

 We will make use of the following estimate of $T^N_n$.

\begin{lemma}{For all $1 \leq n \leq N$
$$
\E T^N_n \leq \frac{2}{\phi_{\min}(C) n}.
$$}
\label{lem:ETNn Bound}
\begin{proof}
At time $T^N_n$ let $k_1, \ldots, k_n$ be the tokens still alive in $C$. As the $N$ token process is Markov, from the memoryless property of exponential random variables we have that conditional on the process up to time $T^N_n$, the later meeting times $t_{k_i, k_j} - T^N_n$ are distributed as independent exponentials of rate $\nu_{k_i, k_j}$. Therefore, conditional on the past, the next meeting $T^N_{n - 1} - T^N_n$ is distributed as an exponential of rate
$$
\theta \geq \sum_{i \neq j \leq n} \nu_{k_i, k_j},
$$
where inequality comes from the rate at which the atoms $k_1, \ldots, k_n$ merge with atoms outside of $C$ that are still alive.

As $\nu_{a,b} \geq \phi_{\min}(C)$ for all $a, b \in C$, we can conclude that $\theta \geq \binom{n}{2} \phi_{\min}(C)$.

This allows us to calculate that
$$
\E T^N_{n-1} - T^N_{n} \leq \frac{2}{\phi_{\min}(C) n(n - 1)},
$$
and so
\begin{align*}
\E T^N_n &= \E \sum_{m = n - 1}^N T^N_m - T^N_{m + 1} \\
&\leq \sum_{m = n}^{N - 1} \frac{2}{\phi_{\min}(C) m(m + 1)} \\
&= \frac{2}{\phi_{\min}(C)}\left( \frac{1}{n} - \frac{1}{N} \right).
\end{align*}
\end{proof}
\end{lemma}

From this we can easily complete our proof of \Cref{prop:K finite for t}.

\begin{proof}
Applying Markov's inequality to \Cref{lem:ETNn Bound}, we find that for any $n \geq 1$
$$
\Prb ( K^N_t(C) \geq n) = \Prb ( T^N_n \geq t) \leq \frac{2}{\phi_{\min}(C) t n}.
$$
As $K^N_t \uparrow K_t$ almost surely, this implies that
$$
\Prb(K_t(C) \geq n) \leq \frac{2}{\phi_{\min}(C) t n}
$$
which easily gives that $K_t(C) < \infty$ almost surely, completing the proof.
\end{proof}

The proof of \Cref{prop:K finite exp} follows by a similar argument, looking at second moment estimates of the stopping times $T^N_n$ (see Section 2.1.2, \cite{berestycki2009recent}).

Applying a standard countable additivity argument to \Cref{prop:K finite for t} at a descending sequence of times $t_i \downarrow 0$ we conclude our proof of \Cref{prop:K finite}.

\section{The Metric Coalescent}
\label{sec:MC}

In this section, we justify calling our constructed process an extension of the Metric Coalescent, by showing that from any finitely supported initial measure $\mu \in P_{\fs}(S)$, the process $\mu_t, t \geq 0$ is distributed as the MC as defined in \Cref{sec:define MC}. Our main result in this section is the following.

\begin{prop}{If $\mu \in P_{fs}(S)$, then the process $\mu_t, t \geq 0$ starting from $\mu_0 = \mu$ is distributed as the Metric Coalescent.}
\label{prop:Is MC}
\end{prop}

In \Cref{sec:Finite Support} we've shown that for an initial measure $\mu \in P_{\cs}(S)$, for all times $t >0$ the process is contained in $P_{\fs}(S)$. Combined with the Markov property, \Cref{prop:Is MC} proves the claim in \Cref{thrm:Main Theorem} that if $\mu \in P_{\cs}(S)$, then for all $t_0 > 0$, the process $\mu_t, t \geq t_0$ is also distributed as the Metric Coalescent.

\subsection{The Partial Order $\prec$}

The evolution of the MC process $\mu_t, t \geq 0$ is defined in terms of the asymmetric description given by the token process, where tokens merge into lower tokens. Viewing the process through the symmetric description
$$
\mu_t = \sum_{k = 1}^{\infty} p_k(t) \delta(s_k),
$$
we no longer know the locations of the tokens and so no longer can determine the future of the process deterministically from the meeting times.

From \Cref{lem:Conditionally Exchg} however, we know how the distribution of the tokens at time $t$ depends upon $\mu_t$ and so know the relative rankings of the lowest elements of the partitions at each atom $s_k$, or equivalently the blocks of the paintbox partition. This is known as the size biased permutation of the paintbox partition.\cite{gnedin1998}

Conceptually we find it simpler to view the size biased permutation as a (random) partial ordering $\prec$ on $S$. Let $\Ord(S)$ be the set of all partial orders on $S$. The order $\prec$ can be defined easily from the initial locations of the tokens $\xi_i, i \geq 1$. For any $s \in S$, write $l(s)$ for the lowest token starting at $s$, i.e.
$$
l(s) = \inf \{i \geq 1 \colon  \xi_i = s \}
$$
with the convention that $l(s) = \infty$ if no such token exists. Then for $s, \tilde{s} \in S$ we define $s \prec \tilde{s}$ if
$$
l(s) < l(\tilde{s}) < \infty.
$$

Note that for all $t > 0$, by \Cref{lem:Support matches standard}
$$
\supp \mu_t \subset \{ \xi_1, \xi_2, \ldots \}
$$
and so we always have that $\prec$ is a full linear order when restricted to $\supp \mu_t$.

Given the initial locations $\xi_i, i \geq 1$ the ordering $\prec$ is fully known. We however are interested in the conditional distribution of $\prec$ given the starting measure $\mu_0 = \mu$. The key motivation here, is that when $\mu_0$ is finitely supported and so meeting rates are discrete, the future distribution of the process will be seen to be determined entirely by just the locations of the atoms of $\mu_t$ and the ordering $\prec$ of the tokens they contain.

\subsection{Size Biased Ordering}

Our first goal in this section is to explicitly calculate the conditional distribution of $\prec$ on $\supp \mu_{t_0}$ at some time $t_0 \geq 0$ given $\mu_t, 0 \leq t \leq t_0$.  We will show that this distribution is given by the size biased ordering, to be defined shortly.

\begin{prop}{The conditional distribution of $\prec$ given the process $\mu_t, 0 \leq t \leq t_0$ is size biased ordering. That is, for any $s_1, \ldots, s_n \in \supp \mu_{t_0}$
$$
\Prb( s_1 \prec \ldots \prec s_n \vert \mu_t, 0 \leq t \leq t_0) = \SBO(\mu_{t_0}, s_1, \ldots, s_n).
$$
}
\label{prop:Prec conditioned}
\end{prop}

\subsubsection{Definition of Size Biased Ordering}
For any $s_1, \ldots, s_n \in S$ and $\nu \in P(S)$ write $\SBO(\nu, s_1, \ldots, s_n)$ as shorthand for the formula 
\begin{align}
SBO(\nu, s_1, \ldots, s_n) = \Prb( \Exp(\nu(s_n)) < \ldots < \Exp(\nu(s_1))), \label{Def SBO}
\end{align}
where here $\Exp(\mu(s_n))$ are independent exponential random variables of rate $\mu(s_n)$. This is one of several equivalent definitions of the size-biased random ordering \cite{gnedin1998} associated with $\nu$. Explicitly this can also be written as
$$
SBO(\nu, s_1, \ldots, s_n) = 
 \frac{1( \nu(s_i) \neq 0, 1 \leq i \leq n) \nu(s_1) \cdots \nu(s_n)}{(\nu(s_1) + \ldots + \nu(s_n))(\nu(s_2) + \ldots) \cdots \nu(s_n)},
$$
but we conceptually prefer the first definition.

We say that a random partial order $\tilde{\prec} \in \Ord(S)$ has $\nu$-\textbf{size biased ordering} if for any $s_1, \ldots, s_n \in S$
$$
\Prb( s_1 \tilde{\prec} \ldots \tilde{\prec} s_n) = SBO(\nu, s_1, \ldots, s_n)
$$

To see that conditionally $\prec$ has size biased ordering, proving \Cref{prop:Prec conditioned}, we need only recall \Cref{lem:Conditionally Exchg} which gives an explicit formula for the conditional location of the tokens. That the ordering then satisfies \Cref{Def SBO} is then standard. \cite{gnedin1998}

\subsection{Proof of \Cref{prop:Is MC}}

We are now able to give a short proof of \Cref{prop:Is MC}.

\begin{proof}
We will prove this inductively over the jumps $T_i$ of $\mu_t, t \geq 0$.

Since $\mu_0 \in P_{\fs}(S)$, each jump clearly occurs when two atoms of $\mu_t$ merge. The meeting rates between two atoms $s_i$ and $s_j$ are given by the meeting rates of the lowest tokens at $s_i$ and $s_j$ - however the rate is $\phi(d(s_i, s_j))$ which clearly depends only on $s_i$ and $s_j$ and not on whatever tokens are located there. This meeting process is obviously Markovian as the meeting rates are exponentially distributed (and so memoryless) and clearly matches the description of the meeting rates of the Metric Coalescent in \Cref{sec:define MC}.

All that's left is to then check the distribution of which token merges with which. Consider the first meeting time $T_1$, between two atoms $s_{i_1}$ and $s_{j_1}$. Which token absorbs the other depends only on their relative ranking, i.e. whether $s_{i_1} \prec s_{j_1}$ or vice versa. By \Cref{prop:Prec conditioned} this is given by
$$
\Prb( s_{i_1} \prec s_{j_1} ) \vert \mu_0) = \frac{\mu_0(s_{i_1})}{\mu_0(s_{i_1}) + \mu_0(s_{j_1})}
$$
matching that of the Metric Coalescent.

Similarly, writing the first $n$ jump (i.e. meeting) times as $T_0 = 0, T_1, \ldots, T_n$,  if at $T_n$ two atoms $s_{i_n}$ and $s_{j_n}$ meet, then the probability of $s_{i_n}$ absorbing $s_{j_n}$ is given by
\begin{align*}
\Prb(s_{i_n} \prec s_{j_n} \vert \mu_{T_{n - 1}}, \ldots, \mu_{T_0} ) &= SBO(\mu_{T_{n -1}}, s_{i_n}, s_{j_n}) \\
&= \frac{\mu_{T_{n - 1}}(s_{i_n})}{\mu_{T_{n - 1}}(s_{i_n}) + \mu_{T_{n - 1}}(s_{j_n})}
\end{align*}
from \Cref{prop:Prec conditioned}. This again clearly matches the Metric Coalescent proving that the two finite jump processes are equal in distribution.
\end{proof}

\section{Martingales of the Metric coalescent}
\label{sec:Martingales}

In this section, we analyse the process $\mu_t, t \geq 0$ by studying a family of associated martingales. In particular, we use martingale arguments to study some of the path properties of our construction. Our main result in this section is the following.

\begin{prop}{From any initial measure $\mu_0 = \mu$ the process $\mu_t$ is right continuous at $t = 0$ almost surely.}
\label{prop:continuous at 0}
\end{prop}

In particular, the same argument (or the Markov property) can be used to show that for any fixed time $t_0$, the MC process $\mu_t, t \geq 0$ is continuous at $t_0$ almost surely. However our proof later that the MC process is cadlag almost surely (see \Cref{sec:Cont and uniq}) only requires right-continuity at time $t = 0$, so we'll focus on that here.

Of course, the process $\mu_t, t \geq 0$ is - for generic initial $\mu \in P(S)$ - not path continuous almost surely. This can easily be seen since from any initial measure in $P_{\fs}(S)$ the process is a discrete jump process and so is never path continuous, only cadlag. For now we'll defer consideration of the path properties of the process until after we prove Feller continuity (see \Cref{sec:TFC}).

To prove \Cref{prop:continuous at 0}, we consider the evolution of functions integrated against our measure-valued process. Specifically, for any $f \in C_b(S)$ we consider the process $\mu_t(f), t \geq 0$ and equivalently $\mu^N_t(f), t \geq 0$ for the $N$-th empirical measures.

Our main result in the study of these processes is that these processes are martingales with explicitly computable quadratic variation.

\begin{prop}{For any $f \in C_b(S)$, $\mu_t(f), t \geq 0$ and $\mu^N_t(f), t \geq 0$ are martingales with quadratic variation satisfying
$$
\E \left( \mu_t(f) - \mu_0(f) \right)^2 = \frac{1}{2} \E \left(1 - \exp( - \phi(d(\xi_1, \xi_2)) t) \right) \left( f(\xi_1) - f(\xi_2) \right)^2
$$
for any time $t \geq 0$.}
\label{prop:M martingale}
\end{prop}

We begin by studying the process dynamics of $\mu^N_t(f)$ and use that to prove the first half of \Cref{prop:M martingale}. Then, a study of the quadratic variation for the case of countably supported initial measures will lead the result for general initial measures.

\subsection{Process Dynamics of the Metric Coalescent}

We begin by showing that the processes $\mu_t^N(f), t \geq 0$ are martingales. We will need to make use of the conditional size biased permutation of the partitions of $\mu^N_t$ at any time $t > 0$, as seen in \Cref{lem:prec on muN}. This simple extension of \Cref{prop:Prec conditioned} follows easily from the fact that the empirical measures $\mu^N_t$ also as evolve as the MC, in this case as the originally defined MC from the (random) initial measure $\mu_0^N$.

Write 
$$
\mu^N_t = \sum_{i =1}^{K^N(t)} p_i \delta(s_i),
$$
where $K^N(t)$ is the finite number of atoms of $\mu^N_t$ at time $t$.

\begin{lemma}{At any fixed time $t \geq 0$, the size biased permutation of $\mu^N_t$ satisfies
$$
\Prb( s_i \prec s_j \vert \mu_t^N) = \frac{p_i}{p_i + p_j}.
$$}
\label{lem:prec on muN}
\end{lemma}

Let $f$ be any bounded continuous function on $S$.

\begin{lemma}{$\mu^N_t(f), t \geq 0$ is a martingale.}
\label{lem:MN Martingale}
\begin{proof}
We show this from the process dynamics of $\mu^N_t(f)$. For some fixed time $t \geq 0$ write
$$
\mu^N_t = \sum_{i =1}^{K^N(t)} p_i \delta(s_i),
$$
with the support of $\mu^N_t$ in any order. For simplicity of notation write 
$$
\nu_{ij} = \phi(d(s_i, s_j))
$$
for the rate of coalescence between the atoms $s_i$ and $s_j$. Then
\begin{align*}
\E \Big( d \mu^N_t(f) &\vert \mu_s, 0 \leq s \leq t \Big) \\
&= \sum_{i \neq j} \nu_{ij} \Big( \frac{p_i}{p_i + p_j} p_j(f(s_i) - f(s_j))
+ \frac{p_j}{p_i + p_j} p_i(f(s_j) - f(s_i)) \Big)\, dt\\
&= 0 \, dt
\end{align*}
since from time $t$, at rate $\nu_{ij}$ the atoms at $s_i$ and $s_j$ merge with - by \Cref{lem:prec on muN} - $s_i$ winning the meeting (i.e. having the lower token) and absorbing the mass $p_j$ of $s_j$ with probability $\frac{p_i}{p_i + p_j}$ and similarly for $s_j$ winning.
\end{proof}
\end{lemma}

As the sequence of martingales $\mu^N_t(f), t \geq 0$ are uniformly bounded since $f$ is, standard convergence results imply that $\mu_t(f), t \geq 0$ is itself a martingale \cite{pollard1984convergence}. Conversely, as the sequence is uniformly integrable this can be shown directly without much difficulty.

To prove the second claim in \Cref{prop:M martingale} for $\mu_t^N$ we begin by analysing the second moment of $\mu_t(f)$ for a countably supported initial measure $\mu$.

\subsection{Second Moment Calculation}
\label{sec:second moment}
Consider a fixed countably supported initial measure
$$
\mu_0 = \sum_{i = 1}^{\infty} p_i(0) \delta (s_i).
$$
By construction, for all times $t \geq 0$ the atoms of $\mu_t$ are contained within those of $\mu_0$ and so we may write
$$
\mu_t = \sum_{i = 1}^{\infty} p_i(t) \delta (s_i),
$$
with $p_i(t)$ random and possibly zero. For each $i \geq 1$, the process $p_i(t), t \geq 0$ is a martingale as can be seen easily (when $\{s_i \}_{i \geq 1}$ is nowhere dense, otherwise by the same argument as for the higher moments) choosing some $f \in C_b(S)$ separating $s_i$ from the other atoms of $\mu_0$. Therefore the expected mass at $s_i$ at time $t$ is given by
$$
\E p_i(t) = p_i(0),
$$
for all $ t \geq 0$. We begin by calculating the second moments of these masses.

\begin{prop}{For any $i \neq j \geq 1$ and $t \geq 0$ we have
$$
\E p_i(t) p_j(t) = p_i(0) p_j(0) \exp( - \phi(d(s_i, s_j)) t ).
$$
Similarly we have
$$
\E p_i^2(t) = \sum_{j \geq 1} p_i(0) p_j(0) \left( 1 - \exp ( - \phi(d(s_i, s_j)) t) \right).
$$
}
\label{prop:second moment calc}
\begin{proof}
First, by the paintbox construction of the token process at time $t$, we have that $p_i(t) p_j(t)$ is the probability of finding any two fixed tokens at $s_i$ and $s_j$ respectively. Thus
$$
\E p_i(t) p_j(t) = \Prb ( \text{ token } 1 \text{ is at } s_i, \text{ token } 2 \text{ is at } s_j \text{ at time } t). 
$$
Now, the only way this can occur is if $\xi_1 = s_i, \xi_2 = s_j$ and tokens $1$ and $2$ haven't met by time $t$, that is $t_{12} \geq t$. By the definition of the token process (or equivalently its paintbox construction at time zero) we can conclude that
\begin{align*}
\E p_i(t) p_j(t) &= \Prb( \xi_1 = s_i, \xi_2 = s_j, t_{12} \geq t) \\
&= p_i(0) p_j(0) \Prb( t_{12} \geq t \vert  \xi_1 = s_i, \xi_2 = s_j ) \\
&= p_i(0) p_j(0) \exp(- \phi(d(s_i, s_j)) t ).
\end{align*}

The second statement follows by a similar proof, noting that if both tokens $1$ and $2$ are at $s_i$ at time $t$, then token $1$ must have started at $s_i$ and their first meeting must satisfy $t_{12} \leq t$.
\end{proof}
\end{prop}

In theory this same idea gives explicit formulas for higher moments of $p_i(t), i \geq 1$, where an $n$-th moment corresponds to a certain subset of meeting trees on $n$ tokens. In practice however, for anything past the second moment the number of possible meeting trees makes such a calculation impractical. We discuss this further in \Cref{sec:moment calc}.

Using \Cref{prop:second moment calc} we are able to calculate the quadratic variation of $\mu_t(f)$ from a countably supported initial measure $\mu_0$.

\begin{lemma}{For any countably supported initial measure
$$
\mu_0 = \sum_{i = 1}^{\infty} p_i(0) \delta(s_i)
$$
and $f \in C_b(S)$, the quadratic variation of $\mu_t(f)$ satisfies
$$
\E \left(\mu_t(f) - \mu_0(f) \right)^2 
= \sum_{i \neq j} p_i(0) p_j(0) \left(1 - \exp( - \phi(d(s_i, s_j)) t) \right) \left( f(s_i) - f(s_j) \right)^2.
$$}
\label{lem:countable 2nd moment}
\begin{proof}
This follows by an application of \Cref{prop:second moment calc} to
$$
\mu_t(f) = \sum_{i \geq 1} p_i(t) f_i(t),
$$
after some simplification, recalling that $\mu_t(f), t \geq 0$ is a martingale and so
$$
\E \left(\mu_t(f) - \mu_0(f) \right)^2 = \E \left( \mu_t(f) \right)^2 - \left( \mu_0(f) \right)^2.
$$
\end{proof}
\end{lemma}

We can now complete our proof of \Cref{prop:M martingale} by applying \Cref{lem:countable 2nd moment} to the sequence of empirical measures $\mu^N_t, N \geq 1$ - which for any initial measure $\mu \in P(S)$ are still countably supported.

\begin{proof}
 Write $F_t \colon S^2 \rightarrow \R$ for the function
$$
F_t(x, y) = \begin{cases}
\left(1 - \exp(-\phi(d(x, y))t) \right)(f(x) - f(y))^2 &\mbox{if } x \neq y \\
0 &\mbox{if } x = y.
\end{cases}
$$
Importantly note that $F_t(x, y)$ is continuous, in particular along the diagonal $x = y$. This is easy to see as $\left(1 - \exp(-\phi(d(x, y))t) \right)$ is bounded and so if for some sequences $x_i, i \geq 1$ and $y_i, i \geq 1$ we have $d(x_i, y_i) \rightarrow 0$ then
$$
(f(x_i) - f(y_i))^2 \rightarrow 0
$$
showing that $F_t(x_i, y_i) \rightarrow 0$. Since by assumption $f \in C_b(S)$ is bounded, so is $F_t(x, y)$ and so clearly $\E F_t(\xi_1, \xi_2)$ is well defined and also bounded.

We claim that
\begin{align}
\E \left( \left( \mu^N_t(f) - \mu^N_0(f) \right)^2 \vert \mu^N_0 \right) 
= \sum_{i, j = 1}^N \frac{1}{2 N^2} F_t(\xi_i, \xi_j). \label{eq:muN moment}
\end{align}

\Cref{lem:countable 2nd moment} gives us a formula for $\E \left( \left( \mu^N_t - \mu^N_0 \right)^2 \vert \mu^N_0 \right)$, however some care is needed to prove \Cref{eq:muN moment} as
$$
\mu^N_0 = \frac{1}{N} \sum_{i = 1}^N \delta(\xi_i)
$$
may not be a atomic decomposition of $\mu_0^N$ as some tokens may already be at the same location at time $t = 0$. To see that \Cref{eq:muN moment} still holds, consider two atoms $p_1 \delta(s_1)$ and $p_2 \delta(s_2)$ of $\mu^N_0$. Then we have that
\begin{align*}
p_1 p_2 F_t(s_1, s_2) &= \left(\sum_{i = 1}^N \frac{1(\xi_i = s_1)}{N} \right)\left(\sum_{i = 1}^N \frac{1(\xi_j = s_2)}{N} \right) F_t(s_1, s_2) \\
&= \sum_{i = 1}^N \sum_{j = 1}^N \frac{1(\xi_i = s_1) 1(\xi_j = s_2)}{N^2} F_t(s_1, s_2) \\
&=\sum_{i = 1}^N \sum_{j = 1}^N \frac{1(\xi_i = s_1) 1(\xi_j = s_2)}{N^2} F_t(\xi_i, \xi_j). 
\end{align*}
Thus writing an atomic decomposition for $\mu^N_0$ as
$$
\mu^N_0 = \sum_{k = 1}^K p_k \delta(s_k)
$$
we have that, summing over all pairs of atoms of $\mu^N_0$ and using that $F_t$ is zero on the diagonal, i.e
$$
F_t(s_{k}, s_{k}) = 0
$$
for all $1 \leq k \leq K$ we have that
\begin{align*}
\E \left( \left( \mu^N_t(f) - \mu^N_0(f) \right)^2 \vert \mu^N_0 \right) 
&= \sum_{k_1 \neq k_2} p_{k_1} p_{k_2} F_t(s_{k_1}, s_{k_2}) \\
&= \frac{1}{2} \sum_{k_1 = 1}^K \sum_{k_2 = 1}^K p_{k_1} p_{k_2} F_t(s_{k_1}, s_{k_2}) \\
&= \frac{1}{2} \sum_{k_1 = 1}^K \sum_{k_2 = 1}^K \sum_{i = 1}^N \sum_{j = 1}^N \frac{1(\xi_i = s_{k_1}) 1(\xi_j = s_{k_2})}{N^2} F_t(\xi_i, \xi_j) \\
&= \frac{1}{2} \sum_{i = 1}^N \sum_{j = 1}^N \frac{1}{N^2}  F_t(\xi_i, \xi_j),
\end{align*}
proving \Cref{eq:muN moment}.

To finish the proof of \Cref{prop:M martingale}, taking the expectation of \Cref{eq:muN moment} using the exchangeability of the initial locations $\xi_i, 1 \leq i \leq N$ we have that
\begin{align*}
\E \left( \mu^N_t(f) - \mu^N_0(f) \right)^2 &= \frac{N^2}{2 N^2} \E F_t(\xi_1, \xi_2) \\
&= \frac{1}{2} \E F_t(\xi_1, \xi_2).
\end{align*}
By \Cref{prop:Token Limit}, $\mu^N_t \rightarrow \mu_t$ weakly almost surely and so $\mu^N_t(f) \rightarrow \mu_t(f)$ almost surely; and similarly for $\mu^N_0$. Therefore as $f$ is bounded we can apply the Bounded Convergence theorem to complete the proof.
\end{proof}

We are now able to complete our proof of \Cref{prop:continuous at 0}.

\subsection{Proof of \Cref{prop:continuous at 0}}

To show the almost sure continuity of $\mu_t, t \geq 0$ at time $t = 0$, we first show the almost sure continuity of $\mu_t(f)$ at $t = 0$ for any $f \in C_b(S)$ and then apply \Cref{thrm:Random Measure Convergence} as was done in the proof of \Cref{prop:Token Limit}.

\begin{lemma}{For a fixed function $f \in C_b(S)$ the process $\mu_t(f), t \geq 0$ is almost surely continuous at $t = 0$.}
\begin{proof}
\label{lem:Mtf continuous at t}
First, by \Cref{prop:M martingale} $\mu_t(f), t \geq 0$ is a bounded martingale. So, applying the martingale convergence theorem for backwards martingales - Theorem 4.6.1, \cite{durrett2010probability} - at $t = 0$ we see that $\mu_t(f), t \geq 0$ has an almost sure limit $M_{\lim}^f$ as $t \downarrow 0$. So, to complete the proof we need to show that $M_{\lim}^f = \mu_0(f)$ almost surely.

The second claim of \Cref{prop:M martingale} tells us that
$$
\E \left(\mu_t(f) - \mu_0(f) \right)^2 = \E F_t(\xi_1, \xi_2)
$$
writing $F_t(x, y)$ as in the proof of \Cref{prop:M martingale}. We claim that
$$
F_t(\xi_1, \xi_2) \rightarrow 0
$$
almost surely as $t \downarrow 0$. To see this, when $\xi_1 = \xi_2$ already we have $F_t(\xi_1, \xi_2) = 0$ as
$$
f(\xi_1) = f(\xi_2).
$$
When $\xi_1 \neq \xi_2$, the limit follows from
$$
1 - \exp( - \phi(d(\xi_1, \xi_2)) t ) \rightarrow 0
$$
and the fact that $f$ is bounded. Thus applying the Bounded Convergence theorem to $F_t(\xi_1, \xi_2)$ we get that
$$
\E \left(\mu_t(f) - \mu_0(f) \right)^2 \rightarrow 0,
$$
as $t \downarrow 0$ showing that $\mu_t(f) \rightarrow \mu_0(f)$ in $L^2$. As we already know that $\mu_t(f), t \geq 0$ has an almost surely limit $M_{\lim}^f$ at $t = 0$, we can conclude that $\mu_0(f) = M_{\lim}^f$ almost surely completing the proof.
\end{proof}
\end{lemma}

To conclude that $\mu_t, t \geq 0$ is right continuous at $t = 0$, we need only recall the standard \Cref{thrm:Test Class} giving a countable convergence determining class  for weak convergence. Therefore, applying \Cref{thrm:Random Measure Convergence} with \Cref{lem:Mtf continuous at t} we complete the proof of \Cref{prop:continuous at 0}.

\section{Feller Continuity}
\label{sec:TFC}

The final step of our proof of existence in \Cref{thrm:Main Theorem} is to show that our extended Metric Coalescent process $\mu_t$ is Feller continuous. Following the terminology of \cite{revuz1999continuous}, write $C_b(P(S))$ for continuous bounded functions on $P(S)$. Then to show that our process satisfies the definition of ``Feller continuous'' we must show that for $F \in C_b(P(S))$ and any $t \geq 0$, the transition operator $P_t$ on $C_b(P(S))$ given by
$$
(P_t F) (x) \E_{\mu_0 = x} F(\mu_t)
$$
is also in $C_b(P(S))$. In our setting, this is equivalent to showing that the random measures $\mu_t$ are weakly continuous, that is, if $\mu^{(i)} \rightarrow \mu$ weakly in $P(S)$, then
$$
\mu^{(i)}_t \rightarrow \mu^{(i)}_t
$$
weakly in $P(P(S))$. Our goal in this section is to prove Feller continuity.

\begin{prop}{The extension of the Metric coalescent process is Feller continuous.}
\label{prop:Feller}
\end{prop}

Note the slight distinction here between a Feller continuous process and the other definition (e.g. \cite{revuz1999continuous}) of a ``Feller process'' i.e. a process generated by a Feller semi-group $P_t, t \geq 0$, whose state space is required to be locally compact, which (generically) does not hold in our setting for $P(S)$. Nonetheless, we will still show that a few properties of Feller processes still hold, in particular the existence of a cadlag version -- in our case the constructed version, see \Cref{prop:cadlag} -- and its subsequent uniqueness in distribution.

In the case when $S$ is compact, the space of Borel probability measures $P(S)$ is compact (and so trivially locally compact) and as an immediate consequence of \Cref{prop:Feller} the MC can be seen to be a Feller process.

\subsection{Convergence of Random Measures}

Our goal in this section is to show the weak convergence of the MC processes for all times $t \geq 0 $ for a sequence of weakly convergent initial measures.

\begin{prop}{Let $\mu^{(i)} \rightarrow \mu$ in $P(S)$ and $t \geq 0$ any fixed time. Then
$$
\mu^{(i)}_t \rightarrow \mu_t
$$
weakly in $P(P(S))$.}
\label{prop:Weak convergence of transition}
\end{prop}

For time $t = 0$, this is trivial, so we focus on proving this for positive times $t$. Our proof relies on the following criteria for weak convergence of random measures (\cite{daley1988introduction}, Theorem 11.1.8), as well as a standard moment method for convergence in distribution of bounded random variables (Section 2.3.e, \cite{durrett2010probability}).

\begin{theorem}{Let $\nu^i, 1 \leq i \leq \infty$ be a sequence of random measures in $P(S)$. Then the following are equivalent:
\begin{enumerate}
\item $\nu^i \rightarrow \nu^{\infty}$ weakly in $P(P(S))$.
\item For all $f \in C_b(S)$ with bounded support,
$$
\nu^i(f) \xrightarrow{d} \nu^{\infty}(f).
$$
\item For all $f \in C_b(S)$ with bounded support and $k \geq 1$
$$
\E ( \nu^i(f) )^{k} \rightarrow \E (\nu^{\infty}(f))^{k}.
$$
\end{enumerate}}
\label{thrm:Convergence Random Measures Equivalence}
\end{theorem}

The outline of our proof of \Cref{prop:Weak convergence of transition} is as follows:  for any fixed function $f \in C_b(S)$, we first reduce the moment calculation to a compactly supported approximation of $f$, and then reduce the comparison between $\mu_t^{(t)}$ and $\mu_t$ to a comparison between their $N$-th token processes. To do this, we give a coupling of the token processes. The key proposition in the proof is the following.

\begin{prop}{Let $\mu^{(i)} \rightarrow \mu$ weakly in $P(S)$. Then for all $t \geq 0, N \geq 1$ and $f \in C_b(S)$, there is an $I(N, f)$ such that if $i \geq I$ then
$$
\inf_{\Gamma} \E_{\Gamma} (\mu^{(i), N}_t(f) - \mu_t^{N}(f) )^2 \leq \frac{f_{\max} + 1}{N},
$$
where the infinum is over all couplings $\Gamma$ of $\mu^{(i), N}_t$ and $\mu_t^N$.}
\label{prop:Token Moment Convergence}
\end{prop}

In fact, our proof of \Cref{prop:Token Moment Convergence} shows that as $i \rightarrow \infty$
$$
\inf_{\Gamma} \E_{\Gamma} (\mu^{(i), N}_t(f) - \mu_t^{N}(f) )^2 \rightarrow 0,
$$
however for notational purposes it is far more convenient (and still sufficient) to leave our result as stated. See \Cref{sec:moment calc} for discussion of an alternative approach to \Cref{prop:Token Moment Convergence}.

A quick note on notation: we will use $\mu^{(i)}$ to refer to the sequence of initial measures and $\mu^{(i)}_t, t \geq 0$ to the associated processes. The corresponding $N$-th token process will be written as $\mu_t^{(i), N}$ for $\mu^{(i)}$ and as usual $\mu^N_t$ for $\mu$.

\subsection{Proof of \Cref{prop:Weak convergence of transition}}

We defer the proof \Cref{prop:Token Moment Convergence} which is rather technical and first complete the proof of \Cref{prop:Weak convergence of transition}. We will need the following convergence bound.

\begin{lemma}{Let $ f \in C_b(S)$ be bounded by $|f| \leq f_{\max}$. Then for any $\nu \in P(S)$ and all $ t \geq 0$
$$
\E | \nu_t^N (f) - \nu_t(f) | \leq \frac{f_{\max}}{\sqrt{N}}.
$$}

\label{lem:Token integral convergence bound}
\begin{proof}
This is essentially an extension of the almost sure convergence shown in \Cref{lem:token limit for f}, using boundedness to prove convergence in $L^2$. Write
$$
\nu_t^N(f) = \frac{1}{N}\sum_{i = 1}^N f_i(t)
$$
where $f_i(t)$ are conditionally independent given $\nu_t$, defined by
$$
f_i(t) = \sum_{i = k}^{\infty} f(s_k(t)) 1( i \in B_k(t)),
$$
and thus bounded by $|f_i(t)| \leq f_{\max}$. Recall from the paintbox partition that $B_k$ is the block of $Z(t)$ corresponding to the atom $s_k$ of $\nu_t$.

Then, by conditional independence we have
$$
\E \left( \frac{1}{N}\sum_{i = 1}^N f_i(t) - \nu_t(f) \right)^2 \leq \frac{f_{\max}^2}{N},
$$
from which our claim follows easily.
\end{proof}
\end{lemma}

 Our general method for proving \Cref{prop:Weak convergence of transition} will be to first reduce the calculation for $f$ to a compactly supported approximation $f_{\epsilon}$, then by the inherent coupling of the token process to the MC reduce to a calculation for $\mu^{(i), N}_t$ and $\mu^N_t$. Finally, applying \Cref{prop:Token Moment Convergence} the result will follow easily.

We make use of the following lemma on compactly supported approximation.

\begin{lemma}{For any $f \in C_b(S)$ and compact $K \subset S$, there exists a compactly supported $\tilde{f} \in C_0(S)$ such that:
\begin{enumerate}
\item $|\tilde{f}(x)| \leq |f(x)|$ for all $x \in S$,
\item $\tilde{f}(x) = f(x)$ for all $x \in K$.
\end{enumerate}}
\label{lem:compact support approx}
\end{lemma}

This is essentially an application of Urysohn's Lemma (Theorem 2.12, \cite{rudin1987real}).

We are now ready to prove \Cref{prop:Weak convergence of transition}.

\begin{proof}
By \Cref{thrm:Convergence Random Measures Equivalence}, it suffices to consider a fixed $f \in C_b(S)$ and $k \geq 1$. We then need to show that
$$
\E \left( \mu_t^{(i)}(f) \right)^{k} \rightarrow \E \left( \mu_t(f) \right)^{k}.
$$
Write $f_{\max}$ for the supremum of $|f|$ on $S$.

First, by assumption $\mu^{(i)} \rightarrow \mu$ and so the measures in $P(S)$ are tight. Thus, for any $\epsilon > 0$ there is a compact $K_{\epsilon} \subset S$ such that
$$
\mu^{(i)}(K_{\epsilon}) \geq 1 - \epsilon
$$
and similarly for $\mu$. Let $f_{\epsilon}$ be the continuous, compactly supported extension of $f$ on $K_{\epsilon}$ guaranteed by \Cref{lem:compact support approx}.

We claim that
\begin{align}
\E | \mu_t(f) - \mu_t(f_{\epsilon}) | \leq 2 f_{\max} \epsilon,
\label{eq:mutf on K}
\end{align}
and similarly for $\mu_t^{(i)}$. This can be seen easily from
\begin{align*}
\E | \mu_t(f) - \mu_t(f_{\epsilon}) | &\leq \E \mu_t(| f - f_{\epsilon}|) \\
&\leq \E \mu_0(|f - f_{\epsilon}|) \\
&= 2 f_{\max} \E \mu_0(S \setminus K_{\epsilon} ) \\
&\leq 2 f_{\max} \epsilon
\end{align*}
using the fact that $\mu_t(|f - f_{\epsilon}|)$ is a bounded martingale (as in \Cref{prop:M martingale}).

Next, we note that for any $X, Y$ with $|X|, |Y| \leq f_{\max}$ that
$$
|X^{k} - Y^{k}| \leq k f_{\max}^{k - 1} |X - Y|
$$
by the simple Lipschitz bound on $f(x) = x^{k}$.

Putting this all together, we can bound
$$
| E (\mu_t(f))^{2k} - \E (\mu_t^{(i)}(f) )^{2k} |
$$
as follows:
\begin{align*}
|\E &(\mu_t(f))^{k} - \E (\mu_t^{(i)}(f) )^{k} | \\
&\leq | \E(\mu_t(f))^{k} - \E (\mu_t(f_{\epsilon}))^{k}| + |\E (\mu_t(f_{\epsilon}))^{k} - \E (\mu^{(i)}_t(f_{\epsilon}))^{k}|
+ |\E (\mu^{i}_t(f_{\epsilon}))^{k} - \E (\mu^{i}_t(f))^{k} | \\
&\leq   \E|(\mu_t(f))^{k} - (\mu_t(f_{\epsilon}))^{k}| + \E| (\mu^{(i)}_t(f)^{k} -(\mu^{(i)}_t(f_{{\epsilon}}))^{k}|
+ | \E (\mu_t(f_{\epsilon}))^{k} - \E (\mu^{(i)}_t(f_{\epsilon}))^{k}| \\
&\leq k f_{\max}^{k - 1} \E \left(|\mu_t(f) - \mu_t(f_{\epsilon})|  + |\mu^{(i)}_t(f) - \mu^{(i)}_t(f_{\epsilon})| \right)
+ | \E (\mu_t(f_{\epsilon}))^{k} - \E (\mu^{(i)}_t(f_{\epsilon}))^{k}| \\
&\leq 4k f_{\max}^{k} \epsilon + \E | (\mu_t(f_{\epsilon}))^{k} - (\mu^N_t(f_{\epsilon}))^{k}| 
+ \E | (\mu^{(i)}_t(f_{\epsilon}))^{k} - (\mu^{(i), N}_t(f_{\epsilon}))^{k}|\\
&+ |\E(\mu^N_t(f_{\epsilon}))^{k} - \E (\mu^{(i), N}_t(f_{\epsilon}))^{k}| \\
&\leq 4k f_{\max}^{k} \epsilon + \frac{2 f_{\max}}{\sqrt{N}} + \inf_{\Gamma} \E |(\mu^N_t(f_{\epsilon}))^{k} - (\mu^{(i), N}_t(f_{\epsilon}))^{k}|
\end{align*}

Now by \Cref{prop:Token Moment Convergence}, since $f_{\epsilon}$ is compactly supported, for all $N$, there exists an $I(N, \epsilon)$ so that if $i \geq I(N, \epsilon)$ then 
$$
\inf_{\Gamma} \E |(\mu^N_t(f_{\epsilon}))^{k} - (\mu^{(i), N}_t(f_{\epsilon}))^{k}| \leq \frac{6 f_{\max} + 1}{N},
$$
giving that for $i \geq I(N, \epsilon)$
$$
|\E (\mu_t(f))^{k} - \E (\mu_t^{(i)}(f) )^{k} | \leq 4k f_{\max}^{k} \epsilon + \frac{2 f_{\max}}{\sqrt{N}} + \frac{6 f_{\max} + 1}{N}.
$$

Choosing $\epsilon$ small enough and $N$ big enough, this can be made arbitrarily small for $i \geq I(N, \epsilon)$, completing the proof.
\end{proof}

All that's left to show in proving \Cref{prop:Weak convergence of transition} is our proof of \Cref{prop:Token Moment Convergence}, which is rather technical. The proof begins by constructing a coupling between $\mu^N_t$ and $\mu_t^{(i), N}$.

\subsection{The Coupling}
\label{sec:the coupling}

To couple the token process for two initial measures $\mu$ and $\nu$, we make use of Strassen's classical representation of the Prokhorov distance, which itself metrizes weak convergence in $P(S)$. That is, the Prokhorov metric $d_P$ on $P(S)$ can be interpreted for two measures $\mu$ and $\nu$ as
$$
d_P(\mu, \nu) = \inf \{ \epsilon \colon \exists \lambda \in \Lambda \text{ s.t. } \Prb_{\lambda}(d(\xi, \tilde{\xi}) > \epsilon) \leq \epsilon \},
$$
where $\Lambda$ is the set of all couplings $\lambda = (\xi, \tilde{\xi})$ with marginals $\mu$ and $\nu$ \cite{strassen1965existence}.

Fix $t > 0, N \geq 1$ and consider two initial measures $\mu$ and $\nu$ with $d_P(\nu, \mu) \leq \epsilon$ for some $\epsilon$ small. Write $\xi_i, 1 \leq i \leq N$ and $t_{ij}, 1 \leq i, j \leq N$ for the initial locations and meeting times of the token process generated from $\mu$. Similarly write $\tilde{\xi}_i, 1 \leq i \leq N$ and $\tilde{t}_{ij}, 1 \leq i, j \leq N$ for the token process generated from $\nu$. We will define a coupling of these to prove \Cref{prop:Token Moment Convergence}.

Intuitively, our coupling pairs the paths of each of $N$ tokens between the two measures. By assumption, we can couple each token's starting location so $\xi_i$ and $\tilde{\xi}_i$ are close with high probability. Then, whenever possible we couple all the meeting times, thus coupling the partition process of the tokens over time. Unfortunately, the inclusion of atomic measures - i.e. meetings at time $t = 0$ - means that a bit of care is needed to carry this out. Heuristically, we finesse this issue by smoothing out meetings at $t = 0$ by coupling them with meetings in some small interval $(0, t_*]$.

We will need the following standard lemma on exponential random variables.

\begin{lemma}{Let $X \sim \exp(a)$ and $Y \sim \exp(b)$ be exponential random variables with rates $a$ and $b$. Then there is a coupling of $X$ and $Y$ so that
$$
\Prb(X \neq Y) \leq 1 - \frac{\min(a, b)}{\max(a, b)}.
$$}
\label{lem:Exp coupling}
\begin{proof}
This follows from the simple calculation of integrating the overlap in the two densities, i.e.
$$
\int_0^{\infty} \min(a \exp(- a t), b \exp(- b t)) dt = \frac{\min(a, b)}{\max(a, b)}.
$$
\end{proof}
\end{lemma}

Our coupling between the token processes $\mu^N_t$ and $\nu^N_t$ is then defined as follows. First, by definition of the Prokhorov distance we may couple $\xi_i$ and $\tilde{\xi}_i$ for each $1 \leq i \leq N$ such that
$$
\Prb (d(\xi_i, \tilde{\xi}_i) > \epsilon ) \leq \epsilon
$$
where $\epsilon \geq d_P(\mu, \nu)$. 

Then, for every pair $i, j$ with both $d(\xi_i, \xi_j) > 0$ and $d(\tilde{\xi}_i, \tilde{\xi}_j) > 0$, we couple the meeting times $t_{ij}$ and $\tilde{t}_{ij}$ as in \Cref{lem:Exp coupling}. In the case $d(\xi_i, \xi_j) = 0$ we have $t_{ij} = 0$ and there is no possibility of a better than independent coupling with (a non-zero) $\tilde{t}_{ij}$. That is, in this case all couplings of $t_{ij}$ with $\tilde{t}_{ij}$ are equivalent. We similarly handle the case $d(\tilde{\xi}_i, \tilde{\xi}_j) = 0$.

\subsection{Proof of \Cref{prop:Token Moment Convergence}}

We being by recalling the (random) equivalence relation $\sim_0$ on $1, 2, \ldots, N$ given by $i \sim_0 j$ if $\xi_i = \xi_j$. For simplicity of notation we'll shorten $\sim_0$ to $\sim$.

We say that an outcome of the coupling is $(\epsilon, d_1, d_2, t_*)$-\textbf{good} for $0 < d_1 < d_2 < \infty$ if
\begin{enumerate}
\item[(G1)] For all $i, j$ either $i \sim j$ or $d(\xi_i, \xi_j) \geq d_1$,
\item[(G2)] For all $i$, $d(\xi_i, \tilde{\xi}_i) \leq \epsilon$,
\item[(G3)] For all $i, j$ s.t. $i \nsim j$ we have $t_{ij} = \tilde{t}_{ij}$,
\item[(G4)] For all $i, j$ s.t. $i \nsim j$ we have $t_{ij} \geq t_* > 0$,
\item[(G5)] For all $i, j$ s.t. $i \sim j$ we have $\tilde{t}_{ij} \leq t_*$,
\item[(G6)] For all $i, j$ we have $d(\xi_i, \xi_j) \leq d_2$ and $d(\tilde{\xi}_i, \tilde{\xi}_j) \leq d_2$.
\end{enumerate}

In words, we allow the possibility that some tokens $i$ and $j$ meet at $t_{ij} = 0$ even if  $\tilde{t}_{ij} \neq 0$. However, we insist that all such meetings $\tilde{t}_{ij}$ happen before $t_*$ and no other meetings happen before $t_*$. As part of this, we require tokens $i$ and $j$ with $\xi_i \neq \xi_j$ to be spaced by at least $d_1$, and require that no pair of tokens for either process are further than $d_2$. Then, our outcome is \textbf{good} if in addition, all the meetings happening after $t_*$ are coupled.

Our motivation for this definition is that in the case that the outcome is $(\epsilon, d_1, d_2, t_*)$-\textbf{good} it becomes simple to bound
$$
|\mu_t^N(f) - \nu_t^{N}(f)|,
$$
for compactly supported $f$. Consider a fixed, compactly supported $f \in C_b(S)$. Then by compactness $f$ must be uniformly continuous, so for every $\alpha > 0$ there is a $\delta_f(\alpha) > 0$ so that for all $x, y \in S$, if $d(x, y) \leq \delta_f(\alpha)$ then
$$
|f(x) - f(y)| \leq \alpha.
$$

\begin{lemma}{If an outcome of the coupling is $(\delta_f(\alpha), d_1, d_2, t_*)$-\textbf{good}, then for $t \geq t_*$
$$
|\mu_t^N(f) - \nu_t^N(f)| \leq \alpha.
$$
\label{lem:exp bound if good}
}
\begin{proof}
First we examine the process at time $t_*$. By $(G2)$, for $\mu^N$ all meetings $t_{ij}$ happen either at time $t = 0$ or at times $t \geq t_*$. By $(G3),(G4)$ and $(G5)$ the only meetings of $\nu^N$ that occur before $t_*$ are those between tokens $i$ and $j$ with $t_{ij} = 0$, and all such meetings occur before $t_*$. 

From this we will be able to show that by time $t_*$, the two partitions processes of tokens are the same. Let $I_i$ be the equivalence class of a token $i$ under $\sim_0$. By definition, then for $\mu^N$ all tokens in $I_i$ meet at time $t = 0$ and have no other meetings until at least time $t_*$ and so
$$
u_i(t_*) = \inf I_i,
$$
since all tokens in $I_i$ are absorbed into token $\inf I_i$ instantly at $t = 0$. Recall here that $u_i(t)$ is the owner of token $i$ at time $t$. Note this trivially true if $i$ is a singleton under $\sim$.

By assumption, for $\nu^N$ all pairs of tokens in $I_i$ meet before $t_*$ and have no meetings with tokens outside of $I_i$ until after $t_*$. These meetings could come in any order but by time $t_*$ all tokens of $\nu^N$ in a class $I_i$ are also owned by $\inf I_i$. That is, we can't say exactly what meeting tree occurs, but we know the final result.

Therefore, for all $1 \leq i \leq N$ we have
\begin{align}
u_i(t_*) = \tilde{u}_i(t_*). \label{eq:u eq}
\end{align}
We claim that \Cref{eq:u eq} also holds for all times $t \geq t_*$. This however follows easily from $(G3)$ as the only meetings for $\mu^N$ that occur after $t_*$ are coupled to the corresponding meetings for $\nu^N$. Thus once the two token partitions are synced at $t_*$, they remain so for all times after.

Therefore, applying $(G2)$ we then have that for all tokens $1 \leq i \leq N$ and times $t \geq t_*$
$$
d(\xi_{u_i(t)}, \tilde{\xi}_{\tilde{u}_i(t)}) \leq \delta(\alpha).
$$
This then implies that
$$
| f(\xi_{u_i(t)}) - f( \tilde{\xi}_{\tilde{u}_i(t)} ) | \leq \alpha,
$$
for $t \geq t_*$.

Recalling \Cref{def:Define muN} we can conclude that for all times $t \geq t_*$
$$
|\mu_t^N(f) - \nu_t^N(f)| \leq \sum_{i = 1}^N \frac{1}{N} | f(\xi_{u_i(t)}) - f( \tilde{\xi}_{\tilde{u}_i(t)} ) | \leq \alpha.
$$
\end{proof}
\end{lemma}

The rest of the difficulty in our proof of \Cref{prop:Token Moment Convergence} is calculating the probability that the outcome is $(\epsilon, d_1, d_2, t_*)$-\textbf{good}, or more specifically bounding the probability that it is not. To start with, we need a technical lemma on the smoothness of the rate function $\phi(x)$. This can be more easily avoided with some smoothness conditions on $\phi$ but we prefer to give a proof in more generality.

\begin{lemma}{Fix $0 < d_1 \leq d_2 < \infty$. There exists a non-decreasing $G_{d_1, d_2} \colon \R_{\geq 0} \rightarrow \R_{\geq 0}$ such that
$$
1 - \frac{\min( \phi(x), \phi(y) )}{\max (\phi(x), \phi(y) )} \leq G_{d_1, d_2}(|x - y|)
$$
for all $x, y \in [d_1, d_2]$ and also $\lim_{z \downarrow 0} G_{d_1, d_2}(z) = 0$.
}
\label{lem:Phi bound}
\begin{proof}
Write $C = [d_1, d_2]^2$ and define
$$
G_{d_1, d_2}(z) = \sup_{(x, y) \in C \colon |x - y| \leq z} 1 - \frac{\min( \phi(x), \phi(y) )}{\max (\phi(x), \phi(y) )}.
$$
Clearly $G_{d_1, d_2}(z)$, as the supremum on a compact set of a function bounded by $1$, is well defined and itself satisfies  $0 \leq G_{d_1, d_2}(z) \leq 1$ for all $z$. Also clear is that $G_{d_1, d_2}(0) = 0$ and $G_{d_1, d_2}(z)$ is non-decreasing in $z$, since as $z$ increases so does the domain of the supremum. So we need only show that $\lim_{z \downarrow 0} G_{d_1, d_2}(z) = 0$.

Assume otherwise and so there exists an $\epsilon > 0$ and a sequence $z_i \downarrow 0$ such that $G_{d_1, d_2}(z_i) \geq \epsilon$. By the compactness of $\{(x, y) \in C \colon |x - y| \leq z \}$ there exists a sequence of points $(x_i, y_i) \in C$ with $|x_i - y_i| \leq z_i$ attaining the supremum, that is
$$
1 - \frac{\min( \phi(x_i), \phi(y_i) )}{\max (\phi(x_i), \phi(y_i) )} = G_{d_1, d_2}(z_i) \geq \epsilon.
$$

By compactness again, there exists a limiting subsequence $(x_j, y_j)$ converging to some point $(x_{\infty}, y_{\infty}) \in C$. As $z_j \downarrow 0$, we must have $|x_j - y_j| \downarrow 0$ and so $|x_{\infty} - y_{\infty}| = 0$, i.e. $x_{\infty} = y_{\infty}$. Since
$$
1 - \frac{\min( \phi(x), \phi(y) )}{\max (\phi(x), \phi(y) )}
$$
is continuous for $(x, y)$ in $C$ - as $\phi$ is bounded away from zero for $x, y \in [d_1, d_2]$ - this implies that
$$
1 - \frac{\min( \phi(x_{\infty}), \phi(y_{\infty}) )}{\max (\phi(x_{\infty}), \phi(y_{\infty}) )} = \lim_j G_{d_1, d_2}(z_j) \geq \epsilon,
$$
clearly a contradiction as it must be $0$.
\end{proof}
\end{lemma}

We'll need to introduce a bit of temporary notation before moving forward. Write
$$
\phi_{\max}(a,b) = \sup_{x \in [a, b]} \phi(x)
$$
which by the continuity of $\phi$ is finite for all $0 < a \leq b < \infty$.

Also write
$$
\phi_{\min}(z) = \inf_{x \in (0, z]} \phi(x)
$$
for $z < \infty$, which by the assumption $\lim_{x \downarrow 0} \phi(x) = \infty$ must be non-zero. Note the difference in the two domains being optimized over.

Also, for $z > 0$, we write
$$
F_{\mu}(z) = \Prb( 0 < d(\xi_1, \xi_2) \leq z),
$$
which by right-continuity must satisfy $F_{\mu}(z) \downarrow 0 $ as $z \downarrow 0$. Also write
$$
\bar{F}_{\mu}(z) = \Prb(z \leq d(\xi_1, \xi_2)),
$$
which satisfies $\bar{F}_{\mu}(z) \downarrow 0$ as $z \uparrow \infty$. Also clearly
$$
F_{\mu}(z) + \bar{F}_{\mu}(z) + \Prb(d(\xi_1, \xi_2) \in \{0, z \}) = 1.
$$

We're finally to state our initial bound on the probability that the coupling is good.

\begin{lemma}{Consider initial measures $\mu$ and $\nu$ with $d_P(\mu, \nu) \leq \frac{\epsilon^2}{2}$. If $\epsilon \leq \frac{d_1}{5}$ then the probability that the outcome is not $(\epsilon, d_1, d_2 t_*)$-\textbf{good} is bounded as follows:
\begin{align*}
\Prb(\text{Not } &(\epsilon, d_1, d_2, t_*)-\text{good}) 
\leq \binom{N}{2} F_{\mu}(d_1) + N \epsilon + \binom{N}{2} G_{\frac{d_1}{2}, d_2}(2 \epsilon) \\
&+ \binom{N}{2} (1 - \exp(- \phi_{\max}(d_1, d_2) t_*)) + \binom{N}{2} \exp( - \phi_{\min}(2\epsilon) t_*) \\
&+ \binom{N}{2}\bar{F}_{\mu}(\frac{d_2}{2}).
\end{align*}}
\label{lem:Pgood bound}
\end{lemma}

To prove \Cref{lem:Pgood bound}, we first bound the probability of failure for each of the requirements for the coupling to be $(\epsilon, d_1, d_2, t_*)$-good. For each of these lemmas, assume that $\mu$ and $\nu$ satisfy the conditions of \Cref{lem:Pgood bound}.

\begin{lemma}{$$
\Prb \left( \text{Not } (G1) \right) \leq \binom{N}{2} F_{\mu}(d_1).
$$}
\label{lem:g1bound}
\begin{proof}
For any pair $i, j$, we have that the probability $(G1)$ fails is
$$
\Prb(  i \nsim j \text{ and } d(\xi_i, \xi_j) < d_1) \leq \Prb( d(\xi_i, \xi_j) \in (0, d_1]) = F_{\mu}(d_1).
$$
The lemma then follows by a simple union bound.
\end{proof}
\end{lemma}

\begin{lemma}{$$
\Prb \left(  \text{Not } (G2) \right) \leq N \epsilon.
$$}
\label{lem:g2bound}
\begin{proof}
For each $i$, by the definition of the coupling we have
$$
\Prb(d(\xi_i, \tilde{\xi}_i) > \epsilon) \leq \epsilon.
$$
Therefore by a simple union bound
$$
\Prb( \exists i \text{ s.t. } (G1) \text{ fails for } i) \leq N \epsilon.
$$
\end{proof}
\end{lemma}

Write $\sigma(\Xi)$ for the sigma field of the initial locations of the tokens, that is
$$
\sigma(\Xi) = \sigma(\xi_i, \tilde{\xi}_i, 1 \leq i \leq N ).
$$

\begin{lemma}{$$
\Prb \left( (G1), (G2), (G6), \text{Not } (G3) \right) \leq \binom{N}{2} G_{\frac{d_1}{2},d_2}(2 \epsilon).
$$}
\label{lem:g3bound}
\begin{proof}
For simplicity of notation, write 
$$
d_{ij} = d(\xi_i, \xi_j), \hspace{ 2 pc } \tilde{d}_{ij} = d(\tilde{\xi}_i, \tilde{\xi}_j).
$$
 First when $(G2)$ holds, by the triangle inequality we always have
$$
|d_{ij} - \tilde{d}_{ij}| \leq d(\xi_i, \tilde{\xi}_i) + d(\xi_j, \tilde{\xi}_j) \leq 2 \epsilon.
$$
When $(G1)$ holds, if $i \nsim j$ then
$$
d_{ij} \geq d_1
$$
and so as $\epsilon \leq \frac{d_1}{5}$
$$
\tilde{d}_{ij} \geq d_1 - 2 \epsilon \geq \frac{d_1}{2},
$$
giving both $d_{ij}$ and $\tilde{d}_{ij}$ at least $\frac{d_1}{2}$.

By \Cref{lem:Exp coupling}, for each such pair $i, j$ the probability that $(G3)$ fails given give the initial locations satisfies
$$
\Prb(t_{ij} \neq \tilde{t}_{ij} \vert \Xi) \leq 1 - \frac{\min(\phi(d_{ij}), \phi(\tilde{d}_{ij}))}{\max(\phi(d_{ij}), \phi(\tilde{d}_{ij}))}.
$$

When $(G6)$ also holds, we have
$$
\frac{d_1}{2} \leq d_{ij}, \tilde{d}_{ij} \leq d_2.
$$
Applying \Cref{lem:Phi bound} with $\frac{d_1}{2}$ and $d_2$, using that $(G1), (G2)$ and $(G6)$ are determined by the initial locations and so in $\sigma(\Xi)$, we have that
\begin{align*}
\Prb \left(t_{ij} \neq \tilde{t}_{ij}, (G1), (G2), (G6)\right) 
&= \E \Prb \left( t_{ij} \neq \tilde{t}_{ij}, (G1), (G2), (G6) \vert \Xi \right)\\
&= \E \left( 1 \left( (G1), (G2), (G6) \right)\Prb(t_{ij} \neq \tilde{t}_{ij}\vert \Xi) \right) \\
&\leq \E \left( 1 \left( (G1), (G2), (G6) \right) \left(  1 - \frac{\min(\phi(d_{ij}), \phi(\tilde{d}_{ij}))}{\max(\phi(d_{ij}), \phi(\tilde{d}_{ij}))}. \right) \right) \\
&\leq \E  \left( 1 \left( (G1), (G2), (G6) \right) G_{\frac{d_1}{2}, d_2}(|d_{ij} - \tilde{d}_{ij} | ) \right) \\
&\leq \E \left(  1 \left( (G1), (G2), (G6) \right) G_{\frac{d_1}{2}, d_2}(2 \epsilon) \right) \\
&\leq G_{\frac{d_1}{2},d_2}(2 \epsilon)
\end{align*}
recalling that $G_{\frac{d_1}{2},d_2}$ is non-decreasing and $|d_{ij} - \tilde{d}_{ij}| \leq 2\epsilon$.

Thus we have that
$$
\Prb \left( (G1), (G2), (G6) \text{ but } (G3) \text{ fails for } i,j \right) \leq G_{\frac{d_1}{2},d_2}(2 \epsilon)
$$
and the proof follows by a simple union bound.
\end{proof}
\end{lemma}

\begin{lemma}{$$
\Prb\left( (G1), (G6), \text{Not } (G4) \right) \leq \binom{N}{2} (1 - \exp(- \phi_{\max}(d_1, d_2) t_*)).
$$
}
\label{lem:g4bound}
\begin{proof}
When $(G1)$ and $(G6)$ hold, if $i \nsim j$ then $d_1 \leq d(\xi_i, \xi_j) \leq d_2$. Therefore by definition of $\phi_{\max}(d_1, d_2)$ we have
$$
\phi(d(\xi_i, \xi_j)) \leq \phi_{\max}(d_1, d_2).
$$
So the probability that $(G4)$ fails for $i,j$ can be calculated, using that $(G1)$ and $(G6)$ are in $\sigma(\Xi)$, and that $t_{ij}$ is (conditionally) exponentially distributed, by
\begin{align*}
\Prb(t_{ij} < t_*, (G1), (G6)) &= \E\left( 1\left((G1), (G6) \right)\Prb (t_{ij} < t_* \vert \Xi) \right) \\
&= \E \left( 1 \left((G1), (G6) \right) \left( 1 - \exp(- \phi(d(\xi_i, \xi_j)) t_*) \right) \right) \\
&\leq \E \left( 1 \left( (G1), (G6) \right) \left( 1 - \exp(- \phi_{\max}(d_1, d_2) t_*) \right)\right) \\
&\leq  1 - \exp(- \phi_{\max}(d_1, d_2) t_*).
\end{align*}

A simple union bound completes the proof.
\end{proof}
\end{lemma}

\begin{lemma}{$$
\Prb \left( (G1), (G2), \text{Not } (G5) \right) \leq \binom{N}{2} \exp( - \phi_{\min}(2\epsilon) t_*).
$$}
\label{lem:g5bound}
\begin{proof}
When $(G1)$ and $(G2)$ hold, if $i \sim j$, using $\epsilon \leq \frac{d_1}{5}$, from the triangle inequality we have 
$$
d(\tilde{\xi}_i, \tilde{\xi}_j) \leq 2 \epsilon.
$$
Therefore 
$$
\phi(d(\tilde{\xi}_i, \tilde{\xi}_j)) \geq \phi_{\min}(2\epsilon),
$$
and so the probability that $(G5)$ fails for $i,j$, using that $(G1)$ and $(G2)$ are in $\sigma(\Xi)$ is
\begin{align*}
\Prb \left( \tilde{t}_{ij} \geq t_*, (G1), (G2) \right) 
&= \E \left( 1 \left( (G1), (G2) \right) \Prb( \tilde{t}_{ij} \geq t_* \vert \Xi) \right) \\
&\leq \E \left( 1 \left((G1), (G2) \right) \exp( - \phi(d(\tilde{\xi}_i, \tilde{\xi}_j)) t_* ) \right)\\
&\leq \exp( - \phi_{\min}(2\epsilon) t_*).
\end{align*}

The lemma then follows from a simple union bound.
\end{proof}
\end{lemma}

\begin{lemma}{$$
\Prb \left((G2), \text{Not } (G6) \right) \leq \frac{N}{2}  \bar{F}_{\mu}(\frac{d_2}{2}).
$$}
\label{lem:g6bound}
\begin{proof}
First, by definition we have that
$$
\Prb(d(\xi_1, \xi_2) \geq \frac{d_2}{2} ) = \bar{F}_{\mu}(\frac{d_2}{2}).
$$
Therefore the probability that
\begin{align}
d(\xi_i, \xi_j) \leq \frac{d_2}{2} \label{eq:dij}
\end{align}
fails for any pair $\xi_i, \xi_j$ can be bounded by the union bound
$$
\frac{N}{2} \bar{F}_{\mu}(\frac{d_2}{2}).
$$
Next, by assumption
$$
d_2 \geq d_1 > 4 \epsilon
$$
 and so if \Cref{eq:dij} holds for all $i$ and $j$, then by the triangle inequality
\begin{align*}
d(\tilde{\xi}_i, \tilde{\xi}_j) &\leq d(\xi_i, \tilde{\xi_i}) + d(\xi_j, \tilde{\xi}_j) + d (\xi_i, \xi_j) \\
&\leq 2\epsilon + \frac{d_2}{2} \\
&\leq d_2
\end{align*}
showing that $(G2)$ holds for all $\tilde{\xi}_i$ and $\tilde{\xi}_j$ completing the proof.
\end{proof}
\end{lemma}

We can now easily complete the proof of \Cref{lem:Pgood bound}.

\begin{proof}
To bound the probability that the outcome of the coupling is not $(\epsilon, d_1, d_2 t_*)$-\textbf{good}, by a union bound we can examine the probability that each of the conditions fails. Since
\begin{align*}
\{ \text{Not Good} \} \subset 
&\{\text{Not } (G1) \} \cup \{\text{Not } (G2) \} \cup \{(G1), (G2), \text{Not } (G5) \} \\
&\cup \{ (G2), \text{Not } (G6) \} \cup  \{(G1), (G2), (G6), \text{Not } (G3) \} \\
&\cup \{(G1), (G6), \text{Not } (G4) \}
\end{align*}
we can combine \Crefrange{lem:g1bound}{lem:g6bound} to complete the proof.
\end{proof}

Note that \Cref{lem:Pgood bound} is nowhere near an optimal bound on the probability the coupling is not good. Our approach here follows the philosophy that since the probability that the coupling isn't $(\epsilon, d_1, d_2, t_*)$-good approaches zero, even a "worst case" bound on it will suffice, which is exactly what happens.

We next need to show that the probability that our outcome is not $(\epsilon, d_1, d_2, t_*)$-good can be made arbitrarily small if we assume $d_P(\mu, \nu)$ is small enough.

\begin{lemma}{There exists a choice of $d_1, d_2, t_*$ and  $\epsilon_1(N, \phi, \mu) \geq 0 $- depending on $N, \phi, \mu$ and $(S, d)$ - such that for any $\epsilon \leq \epsilon_1(N, \phi, \mu)$ and all measures $\nu \in P(S)$, if $d_P(\mu, \nu) \leq \frac{\epsilon^2}{2}$ then
$$
\Prb (\text{Not } (\epsilon, d_*, t_*)-\text{good} ) \leq \frac{6}{N}.
$$}
\label{lem:Pbad bound}
\begin{proof}
We will prove this by showing that there exists a choice of constants $d_1, d_2, t_*$, and $ \epsilon$ - picked in that order - which applied to \Cref{lem:Pgood bound} satisfy our bound.

First, as $\lim_{z \downarrow 0} F_{\mu}(z) = 0$ we may pick $d_1 > 0$ so that
$$
 F_{\mu}(d_1) \leq \frac{1}{N} \binom{N}{2}^{-1},
$$
and thus
\begin{align}
\binom{N}{2} F_{\mu}(d_1) \leq \frac{1}{N}.
\label{eq:peq1}
\end{align}

Next, as $\bar{F}_{\mu}(z) \downarrow 0$ as $z \uparrow \infty$, we can pick $d_2$ large enough so that
$$
\binom{N}{2} \bar{F}_{\mu}(d_2) \leq \frac{1}{N}.
$$

For $t_*$, as $\phi_{\max}(d_1, d_2) < \infty$, we may pick $t_* > 0$ so that
$$
t_* \leq - \frac{\ln \left( 1 - \frac{1}{N \binom{N}{2}} \right)}{\phi_{\max}(d_1, d_2)},
$$
giving that
\begin{align}
\label{eq:peq2}
\binom{N}{2} (1 - \exp( - \phi_{\max}(d_1, d_2) t_*)) \leq \frac{1}{N}.
\end{align}

Next, as $\lim_{x \downarrow 0} \phi(x) = \infty$, there exists a $\gamma$ such that for all $0 < x \leq \gamma$ we have
$$
\phi(x) \geq \frac{\ln(N \binom{N}{2} )}{t_*}.
$$

Similarly, as $\lim_{z \downarrow 0} G_{\frac{d_1}{2}, d_2}(z) = 0$, there exists a $\beta$ so that if $x < \beta$ then
$$
G_{\frac{d_1}{2}, d_2}(2 x) \leq N^{-1} \binom{N}{2}^{-1}.
$$

Therefore, we take
$$
\epsilon_1(N, \phi, \mu) = \min \left(
 \frac{1}{N^2}, \frac{d_1}{5},
 \beta,
  \frac{\gamma}{2} 
  \right).
$$

This will allow us to bound all the terms in \Cref{lem:Pgood bound}. Note that this also easily implies - as $d_1 \leq d_2$ - that $\epsilon_1 \leq \frac{d_2}{4}$, satisfying the minor condition of \Cref{lem:Pgood bound}. First, if $\epsilon \leq \epsilon_1(N, \phi, \mu)$ we easily have 
$$
\epsilon \leq \frac{d_1}{5},
$$
as well as
\begin{align}
\label{eq:peq3}
N \epsilon \leq \frac{1}{N}.
\end{align}

Using $\epsilon \leq \beta$ we have that 
$$
G_{\frac{d_1}{2}, d_2}(2 \epsilon) \leq N^{-1} \binom{N}{2}^{-1}
$$
and so
\begin{align}
\binom{N}{2} G_{\frac{d_1}{2}, d_2}(2 \epsilon) \leq \frac{1}{N}. \label{eq:peq4}
\end{align}
Finally as $2\epsilon \leq \gamma$ we have
$$
\phi(2 \epsilon) \geq \frac{\ln(N \binom{N}{2} )}{t_*}
$$
and so
\begin{align}
\binom{N}{2} \exp(- \phi(2\epsilon)t_*) \leq \frac{1}{N}. \label{eq:peq5}
\end{align}

Applying \Crefrange{eq:peq1}{eq:peq5} to \Cref{lem:Pgood bound}, we see that if $d(\mu, \nu) \leq \frac{\epsilon^2}{2}$ for some $\epsilon \leq \epsilon_1$ then
$$
\Prb(\text{Not } (\epsilon, d_1, d_2, t_*)-\text{good}) \leq \frac{6}{N}.
$$
\end{proof}
\end{lemma}

We are now able finish our proof of \Cref{prop:Token Moment Convergence}.

\begin{proof}
By \Cref{lem:Pbad bound}, there exists a choice of $d_1, d_2, t_*$ and $\epsilon_1(N, \phi, \mu)$ such that for any $\epsilon \leq \epsilon_1(N, \phi, \mu)$, if
$$
d_P(\mu, \nu) \leq \frac{\epsilon^2}{2}
$$
then
$$
\Prb(\text{Not } (\epsilon, d_1, d_2, t_*)-\text{good}) \leq \frac{6}{N}.
$$
Recall $\delta_f(x)$ as in \Cref{lem:exp bound if good} and define $\epsilon_0$ by
$$
\epsilon_0 = \min(\epsilon_1(N, \phi, \mu), \delta_f(\frac{1}{N})).
$$

Therefore, calculating the expectation of
$$
|\mu_t^N(f) - \nu_t^N(f)|
$$
whenever  $d_P(\mu, \nu) \leq \epsilon_0$, using our coupling and \Cref{lem:exp bound if good} with $\epsilon_0 \leq \delta_f(\frac{1}{N})$ we have
\begin{align*}
\E | \mu_t^N(f) - \nu_t^N(f)| &\leq \E | \mu_t^N(f) - \nu_t^N(f)| \left( 1(\text{good}) + 1(\text{Not good} \right) \\
&\leq \frac{1}{N} \Prb(\text{good}) + \E f_{\max} \Prb(\text{Not } (\epsilon, d_1, d_2, t_*)-\text{good}) \\
&\leq \frac{1}{N} + f_{\max}\frac{6}{N} \\
&\leq \frac{6 f_{\max} + 1}{N}
\end{align*}
where $f_{\max}$ is the maximum of $|f|$ on $S$.

This completes the proof, since by assumption $\mu^{(i)} \rightarrow \mu$ weakly and so for $i$ large enough
$$
d_P(\mu^{(i)}, \mu) \leq \epsilon_0.
$$
\end{proof}

\subsection{Moment Calculations for $\mu_t(f)$}
\label{sec:moment calc}

In this section we mention briefly a method that gives an explicit closed form calculation for the moments of $\mu_t(f)$ for any time $t \geq 0$, fixed $f \in C_b(S)$, and initial measure $\mu \in P(S)$. In \Cref{sec:second moment} we calculate the second moment of $\mu_t(f)$ by conditioning on $\mu^N_0$, summing over the different possible meeting trees on $2$ tokens giving the expectation of $\left(\mu^N_t \right)^2$, and then taking the $N \rightarrow \infty$ limit. 

In theory, this same approach gives a perhaps simpler proof of \Cref{prop:Weak convergence of transition} as compared to the coupling method of \Cref{sec:the coupling}, by showing explicitly that if $\mu^{(i)} \rightarrow \mu$ weakly in $P(S)$ then
$$
\E \left(\mu^{(i)}_t(f) \right)^k \rightarrow \E \left( \mu_t(f) \right)^k,
$$
for all $k \geq 1$. The calculation for $k > 2$ follows the same procedure as for the second moment: first conditioning on $\mu^N_0$, then summing over possible meeting trees on $k$ tokens, and finally taking the $N \rightarrow \infty$ limit. In fact, tighter control of the moments $\mu^N_t(f)$ of the $N$ token processes could even provide an alternative proof of the existence of the limit MC process $\mu_t, t \geq 0$ (as in \Cref{prop:Token Limit}) without resort to Kingman's Paintbox theorem.

In practice however, for anything past the second moment this formula appears to be intractable, necessitating our proof of \Cref{prop:Weak convergence of transition} by coupling methods. The main technical issue is two-fold. First, the number of possible meeting trees on $k$ tokens increases approximately as $\binom{k}{2} !$ making an explicit computation (in the form of \Cref{prop:second moment calc}) impossible. More importantly, the (necessary) assumption that
$$
\lim_{x \downarrow 0} \phi(x) = \infty
$$
means that for any fixed meeting tree $T$, the probability that at time $t$ the observed meeting tree $S(t)$ is $S(t) = T$ does not depend continuously (in the weak topology) on the initial measure $\mu$. A posteriori of course, by the coupling argument of \Cref{sec:the coupling}, we know that the moments of $\mu_t(f)$ depend continuously on the initial measure $\mu$ nonetheless. We leave as an open question the details of the pursuit of an alternative approach to \Cref{thrm:Main Theorem} by moment methods.

\section{Continuity and Uniqueness}
\label{sec:Cont and uniq}

In this section we prove first that from any initial measure, the MC process is cadlag almost surely. This builds on \Cref{prop:continuous at 0} where we've shown that the MC process is almost surely cadlag at $t = 0$. Here we're interested in the global path properties of $\mu_t, t \geq 0$.

\begin{prop}{From any initial measure $\mu$, the MC process $\mu_t, t \geq 0$ is cadlag in $P(S)$ with respect to the weak topology, almost surely.}
\label{prop:cadlag}
\end{prop}

From this, we are then able to complete our proof of \Cref{thrm:Main Theorem}, by proving that our extension of the MC process is unique.

\subsection{Proof of \Cref{prop:cadlag}}

Our proof of \Cref{prop:cadlag} begins by considering the Total Variation distance between the measure valued processes $\mu^N_t$ and $\mu^M_t$ for some positive time $t$. Write $d_{\TV}$ for the Total Variation metric on $P(S)$.

\begin{lemma}{Fix $t_0 > 0$ and $M \geq N$. Then
$$
\sup_{t \in [t_0, \infty)} d_{\TV}(\mu^N_t, \mu^M_t) \leq d_{\TV}(\mu^N_{t_0}, \mu^M_{t_0}).
$$}
\label{lem:uniform tvd}
\begin{proof}
The proof follows from the simple idea that the total variation distance can only decrease after each merger (i.e. jump) of the process $\mu^M_t, t \geq t_0$. To see this, write
$$
\mu^M_{t_0} = \sum_{k = 1}^{K^M(t_0)} q_k \delta(s_k)
$$
with the support written in any order. Clearly $\supp \mu^N_{t_0} \subset \supp \mu^M_{t_0}$, since its given by the empirical distribution of the tokens, so we can also write
$$
\mu^N_{t_0} = \sum_{k = 1}^{K^M(t_0)} \tilde{q}_k \delta(s_k)
$$
for some masses $\tilde{q}_k \geq 0$. Note that we assume no particular relation between $q_k$ and $\tilde{q}_k$ - the point is that they are co-located.

The total variation distance of $\mu^N_{t_0}$ and $\mu^M_{t_0}$ is then given by
$$
d_{\TV}(\mu^N_{t_)}, \mu^M_{t_0}) = \frac{1}{2} \sum_{k = 1}^{K^M(t_0)}| q_k - \tilde{q}_k|.
$$

Consider the first jump time $T$ of $\mu^M$. Clearly $\mu^N_t$ can't jump on $[t_0, T)$ if $\mu^M_t$ doesn't. For simplicity of notation assume without loss of generality that at time $T$ the mass at $s_1$ merges into $s_2$. Then
\begin{align*}
2 d_{\TV}(\mu^N_T, \mu^M_T) &= |(q_1 + q_2) - (\tilde{q}_1 + \tilde{q}_2)| + \sum_{k = 3}^{K^M(t_0)} |q_k - \tilde{q}_k| \\
&\leq |q_1 - \tilde{q}_1| + |q_2 + \tilde{q}_2| + \sum_{k = 3}^{K^M(t_0)} |q_k - \tilde{q}_k| \\
&= 2 d_{\TV}(\mu^N_{t_0}, \mu^M_{t_0}).
\end{align*}

As this same argument holds at each of the (finitely many) subsequent jumps, inducting over the meeting times of $\mu^M_t, t \geq t_0$ the proof is complete.
\end{proof}
\end{lemma}

Note that the above proof of \Cref{lem:uniform tvd} holds at $t = 0$, however as there is no longer necessarily a countable support representation of $\mu_0 = \mu$, we don't know if $\mu^N_0, N \geq 1$ is Cauchy (w.r.t. $d_{\TV}$) in $P(S)$. As in \Cref{sec:TVD}, we know $\mu^N_0 \rightarrow \mu_0$ weakly, but for non-atomic $\mu_0$ not in Total Variation distance.

Now, writing $D_I(P(S))$ for the space of cadlag paths from an interval $I$ to $P(S)$, viewed under the uniform (total variation) norm. That is, for any two functions $f, g \in D_I(P(S))$, we write their uniform distance as
$$
d_{\uniform} (f, g) = \sup_{x \in I} d_{\TV}(f(x), g(x)).
$$

Note that typically for a metric space $E$, the space $D_I(E)$ is more naturally considered with respect to the Skorohod metric, which allows paths to "wiggle" in both space and time. However, as $\mu^N_t$ and $\mu_t$ are intrinsically coupled, so are their jump times and so there is no need for the flexibility in time provided by the Skorohod metric. This allows us to use the much stronger topology of uniform convergence, which among other advantages is complete (see Section 3.12 \cite{billingsley2009convergence}).

\begin{lemma}{For any complete metric space $E$ and any interval $I$, the space $D_I(E)$ is complete under the uniform norm.}
\label{lem:uniform conv of cadlag functions}
\end{lemma}

As the empirical measures processes $\mu^N_t, t \geq 0$ are finite jump processes, they are clearly in $D_{I}(P(S))$ for any interval $I$. Recalling \Cref{lem:Convergence in TVD}, \Cref{lem:uniform tvd} implies that for any fixed $t_0 > 0$ the sequence $\mu^N_t, t \geq t_0$ is Cauchy in $D_{[t_0, \infty)}(P(S))$ under the uniform norm. \Cref{lem:uniform conv of cadlag functions} then tells us that the sequence has a cadlag limit.

As we already know (\Cref{prop:Token Limit}) that $\mu_t, t \geq t_0$ is the pointwise limit of the sequence $\mu^N_t, t \geq t_0$, the following lemma is immediate.

\begin{lemma}{For any fixed $t_0 > 0$, $\mu_t, t \geq t_0$ is cadlag almost surely.}
\label{lem:cadlag on bounded intervals}
\end{lemma}

We can now finish our proof of \Cref{prop:cadlag}.

\begin{proof}
Let $A_n$ be the event that $\mu_t, t \geq \frac{1}{N}$ is cadlag. By \Cref{lem:cadlag on bounded intervals}, we have that
$$
\Prb(A_n) = 1.
$$
It then follows that
$$
\Prb( \cap_n A_n) = 1,
$$
and so almost surely $\mu_t$ is cadlag on all of $(0, \infty)$. Since the topology of the Total Variation distance is stronger than the topology of weak convergence, this immediately implies that $\mu_t$ is also cadlag with respect to the weak topology on $P(S)$.

By \Cref{prop:continuous at 0}, we know that $\mu_t$ is also right continuous at $0$ in the weak topology almost surely, completing the proof.
\end{proof}

\subsection{Uniqueness of the Extension}

We are now ready to complete the proof of \Cref{thrm:Main Theorem}. Writing $D_{\R_{\geq 0}}(P(S))$ for the space of cadlag maps from $\R_{\geq 0}$ to $P(S)$. By \Cref{prop:cadlag}, for each initial measure $\mu$ we can think of the MC process as being given by a measure $\Prb_{\mu}$ on $D_{\R}(P(S))$. Our constructed process is then completely determined by this family of measures 
$$
\Prb_{\MC} = \{ \Prb_{\mu}, \mu \in P(S) \}.
$$

To complete our proof of \Cref{thrm:Main Theorem} we need only show that $\Prb_{\MC}$ is the unique such family of measures on $D_{\R}(P(S))$.

\begin{prop}{Any other cadlag, Feller continuous extension of the Metric Coalescent satisfying the properties of \Cref{thrm:Main Theorem} is identically distributed to $\Prb_{\MC}$.}
\end{prop}
\begin{proof}
This follows easily from a standard argument given that our process is cadlag, Markov and Feller continuous. The key idea is that any cadlag time-homogeneous Markov process is separable, i.e. determined uniquely by its finite dimensional distributions (fdd's). The Chapman-Kolmogorov equation show that the fdd's are determined by the single dimensional distributions
$$
P_t(x, A) = \Prb( \mu_t \in A \vert \mu_0 = x),
$$
for $A \subset P(P(S))$. Feller continuity implies then that if $\mu^{(i)} \rightarrow \mu$ so do the single dimensional distributions and therefore so do all the fdd's of the processes.

Since any other extension of the MC agrees with $\Prb_{\MC}$ on all of $P_{\fs}(S)$ which is dense in $P(S)$ under the weak topology, we can conclude that any other such extension has the same fdd's and is therefore identically distributed on $D_{\R}(P(S))$. 
\end{proof}

\section{Some Motivating Examples}
\label{sec:Examples}

Having finished our proof of \Cref{thrm:Main Theorem}, in this section we look at a few particular examples of the Metric Coalescent that motivate some of the assumptions we have made. In particular, we show that without assumptions on the rate function $\phi(x)$, even in the world of finitely supported measures, the Metric Coalescent would not be Feller continuous with respect to the weak topology. We'll also give a construction showing that even for countably supported measures in $P(S)$, without the assumption of compact support, the MC process need not be in $P_{\fs}(S)$ for positive times.

\subsection{The Rate $\phi(x)$ Near Zero}

Much of the technical complexity in the proof of \Cref{thrm:Main Theorem} is a result of $\phi(x)$ being unbounded at $x = 0$. A simple example however shows this to be necessary for Feller continuity to hold. Intuitively, the issue can be seen by looking at a converging sequence $x_i \rightarrow x_{\infty}$ and measures
$$
\mu^{(i)} = \frac{1}{2} \delta(x_i) + \frac{1}{2} \delta(x_{\infty})
$$
and $\mu^{\infty} = \delta(x_{\infty})$. Clearly $\mu^{(i)} \rightarrow \mu^{\infty}$ weakly, however if $\phi(d(x_i, x_{\infty}))$ doesn't increase to $\infty$, then the time between the two atoms merging (for some sub-sequence) is bounded below.

To actually construct a working counter-example, we need to consider at least three atoms.

\subsubsection{An Example}

For ease of notation we work on $S = [0, 1]$ with the standard metric - however this is by no means necessary. We will construct a sequence of initial measures $\mu^{(n)} \rightarrow \mu^{\infty}$ and a continuous function $f \colon [0, 1] \rightarrow \R$ such that
$$
\mu^{(n)}_t(f) \nRightarrow \mu^{\infty}_t(f),
$$
for $t > 0$ whenever $\lim_{x \downarrow 0} \phi(x) \neq \infty$.

 Let $x_n \downarrow 0$ be a sequence of points with the assumption - again purely for simplicity of notation - that $0 \leq x_n \leq \frac{1}{2}$ and consider the sequence of measures
$$
\mu^{(n)} = \frac{1}{3}\left(\delta(0) + \delta(1) + \delta(x_n) \right).
$$
Clearly the sequence $\mu^n$ limits weakly to 
$$
\mu^{\infty} = \frac{2}{3} \delta(0) + \frac{1}{3} \delta(1).
$$

Let $f(x)$ be any continuous function supported near $1$, for instance
$$
f(x) = \begin{cases}
0 &\mbox{ if } 0 \leq x \leq \frac{1}{2}, \\
2x - 1 &\mbox{ if } \frac{1}{2} \leq x \leq 1.
\end{cases}
$$
As in \Cref{sec:Martingales} we will consider the martingales $\mu_t^{(n)}(f)$ and similarly $\mu_t^{\infty}(f)$.

First consider $\mu_t^{\infty}(f)$. Let $T_{0,1}$ be the first meeting time of the atoms at $0$ and $1$, which is distributed as a rate $\phi(1)$ exponential. Before time $T_{0,1}$, we have $\mu_t^{\infty}(f) = \frac{1}{3}$. After time $T_{0, 1}$, with probability $\frac{2}{3}$, the atom at $0$ wins and $\mu_t^{\infty}(f) = 0$. Otherwise, with probability $\frac{1}{3}$, the atom at $1$ wins and $\mu_t^{\infty}(f) = 1$. Thus, the distribution of $\mu_t^{\infty}(f)$ on $\R$ - supported at $\{0, \frac{1}{3}, 1 \}$ - is given by
$$
\exp(- \phi(1) t) \delta(\frac{1}{3}) + \frac{1}{3}(1 - \exp( - \phi(1) t)) \delta (1) + \frac{2}{3}(1 - \exp( - \phi(1) t)) \delta (0),
$$
where here $\delta(x)$ is the point mass at $x$. Most importantly, note that $\mu_t^{\infty}(f)$ has no support at $\frac{2}{3}$.

Next, consider the martingales $\mu^{(n)}_t(f)$. Along with $T_{0, 1}$, write $T_{0, x_n}$ and $T_{x_n, 1}$ for first meeting times between the other two pairs of atoms of $\mu^{(n)}$. We will show that for all $n$ and times $t > 0$, the distribution of $\mu^{(n)}_t(f)$ has non-zero (and bounded below) support at $\frac{2}{3}$.

One way the outcome $\mu^{(n)}_t(f) = \frac{2}{3}$ occurs is if first meeting before time $t$ is $T_{0,1}$, with the atom at $1$ winning, and then $T_{x_n, 1}$ not happening until after $t$. Using standard properties of independent exponential random variables, the chance of this is given by
\begin{align*}
\Prb( T_{0,1} &= \min(T_{0, 1}, T_{0, x_n}) \leq t \leq T_{x_n, 1}  \text{ and  1 wins } T_{0,1}) \\
&= \frac{1}{2} \Prb( T_{0,1} = \min(T_{0, 1}, T_{0, x_n})) \Prb(\min(T_{0, 1}, T_{0, x_n}) \leq t) \Prb( T_{x_n, 1} \geq t) \\
&= \frac{1}{2} \frac{\phi(1)}{\phi(1) + \phi(x_n)} \left(1 - \exp( -(\phi(1) + \phi(x_n))t) \right) \exp(- \phi(1 - x_n) t).
\end{align*}

By continuity $\phi(1 - x_n) \rightarrow \phi(1)$ and clearly $\exp(- \phi(x_n) t) \leq 1$ and so
$$
\liminf_n \Prb(\mu_t^{(n)}(f) = \frac{2}{3}) \geq \liminf_n \frac{1}{2} \frac{\phi(1)}{\phi(1) + \phi(x_n)}(1 - \exp( - \phi(1) t)) \exp( - \phi(1)t).
$$

Therefore, if for our sequence $x_n$ we have $\limsup_{n} \phi(x_n) \neq \infty$, then 
$$
\liminf_n \Prb(\mu^{(n)}_t(f) = \frac{2}{3}) > 0,$$ which implies that
$$
\mu^{(n)}_t(f) \nRightarrow \mu_t^{\infty}(f)
$$
as each of $\mu_t^{(n)}(f), n \geq 1$ and $\mu_t(f)$ have discrete finite support contained within $\{ 0, \frac{1}{3}, \frac{2}{3}, 1 \}$.

Note that this counter-example applies as long as there is any sequence of points $x_n \downarrow 0$ with $\lim_n \phi(x_n) \neq \infty$. Therefore, for the MC process to be Feller continuous the rate function $\phi(x)$ must satisfy
$$
\lim_{x \downarrow 0} \phi(x) = \infty.
$$

\subsection{Finite Support for Positive Times}
\label{sec:Finite Supp Ex}

In \Cref{sec:Finite Support} we showed that if our initial measure $\mu_0 \in P_{\cs}(S)$ then for all positive times $t > 0$, $\mu_t \in P_{fs}(S)$ and the process evolves as the Metric Coalescent. Our proof, via a comparison to Kingman's Coalescent, relies heavily on bounding the meeting rates of atoms away from zero using compactness. In the case of
$$
\inf_{x > 0} \phi(x) > 0,
$$
the same comparison to Kingman's Coalescent holds.

\begin{prop}{If $\inf_{x > 0} \phi(x) > 0$, then for any $\mu \in P(S)$, $\mu_t \in P_{\fs}(S)$ for all $t > 0$, almost surely.}
\end{prop}

Perhaps ideally this would still hold even without such an assumption, however as we'll see that isn't the case. For the rest of this section, we assume that instead
$$
\liminf_{x \rightarrow \infty} \phi(x) = 0.
$$
We'll then give an example of a countably supported initial measure $\mu$, satisfying any moment condition we'd like, which for all positive times $t > 0$ has non-zero probability of not being finitely supported.

\subsubsection{An Example}

For simplicity we'll work over $S = \R$ with the Euclidean metric and assume that $ \phi(x) \rightarrow 0$ as $x \rightarrow \infty$. We can then construct an initial measure $\mu$ dispersed enough so that $K(t)$, defined as the cardinality of the support of the infinite token process, has
$$
K(t) = \infty
$$
for some positive times $t > 0$.

Let $r_i \geq 0, i \geq 1$ be a -- to be specified -- decreasing sequence. We claim that we can then select a sequence of points $s_i, i \geq 1$ in $\R$ such that for each $i$ and $j \neq i$, we have
$$
\phi ( d (s_i, s_j) ) \leq r_i.
$$

We choose our initial measure $\mu$ to have countable support given by these points, i.e.
$$
\mu = \sum_{i = 1}^{\infty} m_i \delta(s_i),
$$
where
$$
\sum_{i = 1}^{\infty} m_i = 1.
$$
We make no assumptions whatsoever on the mass at each point $s_i$ other than $m_i > 0$.

Initially, as all points $s_i \in \supp \mu$ have positive mass, almost surely there is at least one token (actually infinitely many) with initial location $s_i$. At time $t = 0$ all the tokens at each point $s_i$ merge into the lowest there. Let $l(i)$ be the lowest token at $s_i$. Now the total rate $q_i$ at which $l(i)$ merges is clearly bounded by
\begin{align*}
q_i &\leq \sum_{j = 1}^{\infty} \phi(d(s_i, s_j)) \\
&\leq \sum_{j = 1}^{i - 1} r_i + \sum_{j = i + 1}^{\infty} r_j \\
&\leq (i - 1) r_i +  \sum_{j = i + 1}^{\infty} r_j,
\end{align*}
and is in fact much lower since only finitely many (in fact $l(i) - 1$) of these corresponding merges are actually possibly.

Therefore, the total rate $q$ at which the next meeting occurs is bounded by
\begin{align*}
q &\leq \sum_{i = 1}^{\infty} q_i \\
&\leq \sum_{i = 1}^{\infty} \left( (i - 1) r_i +  \sum_{j = i + 1}^{\infty} r_j \right) \\
&= \sum_{i = 1}^{\infty} 2(i - 1) r_i.
\end{align*}

Now we can clearly choose $r_i$ so that the total meeting rate $q$ is finite, for instance $r_i = \frac{1}{i^3}$. Therefore, choosing such an $r_i$, for all times $t > 0$ there is a positive (i.e. non-zero) probability that $K(t) = \infty$. In fact, we have shown something much stronger - that there is a positive probability of \textbf{no} meetings occurring before any time $t$.

Note something else important about this example: it depends only on the support of $\mu$, not in any way on the distribution of mass of $\mu$. Therefore we can choose the mass $m_i$ so that $\mu$ satisfies any moment/concentration condition, for instance
$$
\E \exp( d(\xi_1, \xi_2) t) < \infty
$$
for all $t$, where $\xi_1, \xi_2 \sim \mu$. 

\section{Further Directions}
\label{sec:Open Probs}

In this section we present some possible further directions of research on the Metric Coalescent process.

\subsection{Coming Down From Infinity}

Consider a fixed metric space $(S, d)$ and rate function $\phi$. Our primary interest in the extended MC process is from its connection to the finite MC process, and so its reasonable to consider for which initial measures this connection holds.

In \Cref{sec:Finite Support} we have seen that for compactly supported initial measures $\mu \in P_{\cs(S)}$ that for all $ t > 0$, $\mu_t$ is finitely supported almost surely. In \Cref{sec:Finite Supp Ex} we've seen that even for countably supported (but non-compact) measures, this need not be true.

In the world of $\Lambda$-Coalescents, a class including Kingman's Coalescent, under some mild assumptions there is a zero-one law stating that a coalescent has either finitely many blocks for all $t > 0$ almost surely, corresponding to
$$
\int_0^1 x^{-1} \Lambda(dx) < \infty,
$$
or in the opposite case infinitely many blocks for all $t > 0$ almost surely (Theorem 3.6, \cite{berestycki2009recent}). Here $\Lambda$ is the finite measure on $[0,1]$ generating the coalescent.

In the world of the Metric Coalescent and token process, the situation is a bit more complicated as there is a clear dependence between when partition blocks merge and their location. At any positive time $ t > 0$, it seems likely that the location of the remaining blocks are not independently distributed as $\mu$, since there is a clear bias for "far away" blocks to persist. Nevertheless, we conjecture that an analogous zero-one law holds.

\begin{open}{Is there an analogous zero-one law for the Metric Coalsecent, i.e., is it true that for any initial measure $\mu \in P(S)$, that one of the following holds:
\begin{enumerate}
\item Almost surely, for all $ t \geq 0 $, $\mu_t \notin P_{\fs}(S)$.
\item Almost surely, for all $t > 0$, $\mu_t \in P_{\fs}(S)$.
\end{enumerate}}
\end{open}

If such a zero-one law holds, a natural next step is to consider the class of \textbf{finite type} measures $P_{\finitetype}(S) \subset P(S)$ given by
$$
P_{\finitetype}(S) = \{ \mu \colon \text{ almost surely }, \forall t > 0,  \mu_t \in P_{\fs} \}.
$$
We have shown that $P_{\cs}(S) \subset P_{\finitetype}(S)$, but that not all countably supported measures are finite type.

\begin{open}{Characterize the class $P_{\finitetype}(S)$ of finite type measures.}
\end{open}

\subsection{Time Reversal}

A classical result about Kingman's Coalescent is its duality under a time reversal to a conditioned Yule process\cite{berestycki2009kingman}. Viewing the Metric Coalescent as a generalization of Kingman's model, it's natural then to consider what can be said about the time reversed Metric Coalescent process.

For a compactly supported initial measure $\mu \in P_{\cs}(S)$ we know (by \Cref{sec:Finite Support})  - writing $T_n$ for the first time that $\mu_t, t \geq 0$ has $n$ atoms - that $T_n < \infty$ almost surely for all $n$. Then, the time reversed process
$$
\mu_{T_1}, \mu_{T_2}, \ldots
$$
is a measure-valued branching process in $S$ with initial condition $\mu_{T_1} = \delta(\xi_1)$, i.e. a point mass at a location chosen from $\mu$. Now by \Cref{prop:continuous at 0} we know that
$$
\mu_{T_n} \rightarrow \mu
$$
weakly almost surely. This suggests that the time reversed process can be viewed as a branching process in $S$ conditioned to converge to $\mu$.

\begin{open}{Give an explicit description of the time reversed Metric Coalescent process.}
\end{open}

For a local description of the time reversal at $t = 0$, consider the MC on $S = [0, 1]^2$ started from an initial measure $\mu$ that's absolutely continuous with respect to Lebesgue measure. We can then view the MC process $\mu_t, t \geq 0$ as a point process on $S \times (0, \infty)$ with the point $(s, x)$ representing an atom $x \delta(s)$. Then, analogous to the case of the self-similar $t \rightarrow \infty$ asymptotics in the setting of the classical Smoluchowski coagulation equation \cite{aldous1999deterministic}\cite{fournier2005existence}, we make the following conjecture.

\begin{open}{For some scaling function $\psi(t) \uparrow \infty$ as $t \downarrow 0$, the point process in a shrinking window around some fixed $s_0$, rescaled by the map
$$
(s_0 + s, x) \rightarrow (s \psi(t), x \psi^2(t) )
$$
converges as $t \downarrow 0$ to a translation invariant Poisson Point process on $\R^2 \times (0, \infty)$ of some intensity $\delta(s_0, x)$ which is given by the solution of a certain equation.}
\end{open}

\bibliographystyle{plain}

\bibliography{research}

\begin{thebibliography}{10}

\bibitem{rudin1987real}
{\em Real and complex analysis}.
\newblock McGraw-Hill Book Co., New York, third edition, 1987.

\bibitem{donnelly1999particle}
Particle representations for measure-valued population models.
\newblock {\em Ann. Probab.}, 27(1):166--205, 1999.

\bibitem{ALS14}
David Aldous, Daniel Lanoue, and Justin Salez.
\newblock The compulsive gambler process.
\newblock In preparation, 2014.

\bibitem{aldous1999deterministic}
David~J. Aldous.
\newblock Deterministic and stochastic models for coalescence (aggregation and
  coagulation): a review of the mean-field theory for probabilists.
\newblock {\em Bernoulli}, 5(1):3--48, 1999.

\bibitem{berestycki2009kingman}
Julien Berestycki and Nathana{\"e}l Berestycki.
\newblock Kingman's coalescent and {B}rownian motion.
\newblock {\em ALEA Lat. Am. J. Probab. Math. Stat.}, 6:239--259, 2009.

\bibitem{berestycki2009recent}
Nathana{\"e}l Berestycki.
\newblock Recent progress in coalescent theory.
\newblock 16:193, 2009.

\bibitem{billingsley2009convergence}
Patrick Billingsley.
\newblock {\em Convergence of probability measures}.
\newblock Wiley Series in Probability and Statistics: Probability and
  Statistics. John Wiley \& Sons, Inc., New York, second edition, 1999.
\newblock A Wiley-Interscience Publication.

\bibitem{boissard2011mean}
Emmanuel Boissard and Thibaut Le~Gouic.
\newblock On the mean speed of convergence of empirical and occupation measures
  in wasserstein distance.
\newblock {\em Annales de l'Institut Henri Poincar{\'e}, Probabilit{\'e}s et
  Statistiques}, 50(2):539--563, 05 2014.

\bibitem{daley1988introduction}
D.~J. Daley and D.~Vere-Jones.
\newblock {\em An introduction to the theory of point processes. {V}ol. {II}}.
\newblock Probability and its Applications (New York). Springer, New York,
  second edition, 2008.
\newblock General theory and structure.

\bibitem{durrett2010probability}
Rick Durrett.
\newblock {\em Probability: theory and examples}.
\newblock Cambridge Series in Statistical and Probabilistic Mathematics.
  Cambridge University Press, Cambridge, fourth edition, 2010.

\bibitem{fournier2005existence}
Nicolas Fournier and Philippe Lauren{\c{c}}ot.
\newblock Existence of self-similar solutions to {S}moluchowski's coagulation
  equation.
\newblock {\em Comm. Math. Phys.}, 256(3):589--609, 2005.

\bibitem{gibbs2002choosing}
Alison~L. Gibbs and Francis~Edward Su.
\newblock On choosing and bounding probability metrics.
\newblock {\em International Statistical Review}, 70(3):419--435, 2002.

\bibitem{gnedin1998}
Alexander~V. Gnedin.
\newblock On convergence and extensions of size-biased permutations.
\newblock {\em J. Appl. Probab.}, 35(3):642--650, 1998.

\bibitem{Horowitz1994261}
Joseph Horowitz and Rajeeva~L. Karandikar.
\newblock Mean rates of convergence of empirical measures in the {W}asserstein
  metric.
\newblock {\em J. Comput. Appl. Math.}, 55(3):261--273, 1994.

\bibitem{pollard1984convergence}
David Pollard.
\newblock {\em Convergence of stochastic processes}.
\newblock Springer Series in Statistics. Springer-Verlag, New York, 1984.

\bibitem{prokhorov1956convergence}
Yu.~V. Prohorov.
\newblock Convergence of random processes and limit theorems in probability
  theory.
\newblock {\em Teor. Veroyatnost. i Primenen.}, 1:177--238, 1956.

\bibitem{revuz1999continuous}
Daniel Revuz and Marc Yor.
\newblock {\em Continuous martingales and {B}rownian motion}, volume 293 of
  {\em Grundlehren der Mathematischen Wissenschaften [Fundamental Principles of
  Mathematical Sciences]}.
\newblock Springer-Verlag, Berlin, third edition, 1999.

\bibitem{strassen1965existence}
V.~Strassen.
\newblock The existence of probability measures with given marginals.
\newblock {\em Ann. Math. Statist.}, 36:423--439, 1965.

\bibitem{vershik2004}
A.~M. Vershik.
\newblock Random metric spaces and universality.
\newblock {\em Uspekhi Mat. Nauk}, 59(2(356)):65--104, 2004.

\end{thebibliography}

\end{document}